\documentclass{amsart}%
\usepackage{palatino, mathpazo}
\usepackage{amsfonts}
\usepackage{amsmath}
\usepackage{amssymb,latexsym,xcolor}
\usepackage{graphicx}
\usepackage[mathscr]{eucal}
\usepackage{harpoon}
\usepackage{amssymb}
\providecommand{\U}[1]{\protect \rule{.1in}{.1in}}

\newtheorem{theorem}{Theorem}[section]

\newtheorem{corollary}[theorem]{Corollary}
\newtheorem{definition}[theorem]{Definition}
\newtheorem{example}[theorem]{Example}
\newtheorem{Lemma}[theorem]{Lemma}
\newtheorem{Proposition}[theorem]{Proposition}
\newtheorem{Theorem}{Theorem}

\theoremstyle{remark}
\newtheorem{remark}[theorem]{Remark}

\numberwithin{equation}{section}

\def\N{\mathbb N}

\def\R{\mathbb R}

\def\al{\alpha}

\def\dl{\Delta}
\def\e{\varepsilon}
\def\om{\Omega}
\def\iy{\infty}

\def\la{\lambda}

\def\pa{\partial}

\def\rt{\rightarrow}

\def\lan{\langle}
\def\ran{\rangle}

\def\bt{\begin{theorem}}
\def\et{\end{theorem}}
\def\bl{\begin{Lemma}}
\def\el{\end{Lemma}}
\def\bd{\begin{definition}}
\def\ed{\end{definition}}
\def\bc{\begin{corollary}}
\def\ec{\end{corollary}}
\def\bprop{\begin{Proposition}}
\def\eprop{\end{Proposition}}
\def\bp{\begin{proof}}
\def\ep{\end{proof}}
\def\bx{\begin{example}}
\def\ex{\end{example}}
\def\br{\begin{remark}}
\def\er{\end{remark}}
\def\be{\begin{equation}}
\def\ee{\end{equation}}
\def\bal{\begin{align}}
\def\bn{\begin{enumerate}}
\def\en{\end{enumerate}}
\def\eal{\end{align}}

\def\bg{\begin{align*}}
\def\eg{\end{align*}}
\def\bcs{\begin{cases}}
\def\ecs{\end{cases}}
\def\abs#1{\lvert#1\rvert}
\def\RNum#1{\uppercase\expandafter{\romannumeral #1\relax}}

\def\RN{\mathbb R^N}

\def\bean{\begin{eqnarray*}}
\def\eean{\end{eqnarray*}}

\def\p{\mathcal P}
\def\s{\mathcal S}
\def\mc{\mathcal}

\begin{document}
\title[Biharmonic Lane-Emden Problems]{Asymptotic behavior of least energy nodal solutions for biharmonic Lane-Emden problems in dimension four}
\author{Zhijie Chen}
\address{Department of Mathematical Sciences, Yau Mathematical Sciences Center,
Tsinghua University, Beijing, 100084, China}
\email{zjchen2016@tsinghua.edu.cn}
\author{Zetao Cheng}
\address{Department of Mathematical Sciences,
Tsinghua University, Beijing, 100084, China}
\email{chengzt20@mails.tsinghua.edu.cn}
\author{Hanqing Zhao}
\address{Department of Mathematical Sciences,
Tsinghua University, Beijing, 100084, China}
\email{zhq20@mails.tsinghua.edu.cn}




\begin{abstract}
In this paper, we study the asymptotic behavior of least energy nodal solutions $u_p(x)$ to the following fourth-order elliptic problem
\[
	 \begin{cases}
\Delta^2 u =|u|^{p-1}u \quad &\hbox{in}\;\Omega, \\
u=\frac{\partial u}{\partial \nu}=0 \ \ &\hbox{on}\;\partial\Omega,
\end{cases}
\]
where $\Omega$ is a bounded $C^{4,\alpha}$ domain in $\mathbb{R}^4$ and $p>1$. Among other things, we show that up to a subsequence of $p\to+\infty$, $pu_p(x)\to 64\pi^2\sqrt{e}(G(x,x^+)-G(x,x^-))$, where $x^+\neq x^-\in \Omega$ and $G(x,y)$ is the corresponding Green function of $\Delta^2$.
This generalize those results for $-\Delta u=|u|^{p-1}u$ in dimension two by (Grossi-Grumiau-Pacella, Ann.I.H.Poincar\'{e}-AN, 30 (2013), 121-140) to the biharmonic case, and also gives an alternative proof of Grossi-Grumiau-Pacella's results without assuming their comparable condition $p(\|u_p^+\|_{\infty}-\|u_p^-\|_{\infty})=O(1)$.
\end{abstract}

\maketitle


\section{Introduction}
\label{section-1}

Fourth-order elliptic equations arise in differential geometry and statistical mechanics, and have received ever-increasing interests in the last decades; see e.g. \cite{3,5,11,13,14,18,22,23} and references therein.
Let $\om$ be a bounded $C^{4,\alpha}$ domain in $\R^4$ and $p>1$. In this paper, we study the asymptotic behavior of least energy nodal solutions to the following biharmonic equation with Dirichlet boundary conditions
\begin{equation}\label{PP}
\begin{cases}
\Delta^2 u =|u|^{p-1}u \quad &\hbox{in}\;\Omega, \\
u=\frac{\partial u}{\partial \nu}=0 \ \ &\hbox{on}\;\partial\Omega,
\end{cases}
\end{equation}
as $p\rt+\iy$, where $\nu$ denotes the outer normal vector of $\partial\Omega$.

First let us recall some known results about \eqref{PP}.
Clearly solutions of \eqref{PP} are critical points of the functional $J_p(u): H_0^2(\Omega)\rightarrow \R$,
\begin{equation} J_p(u):=\frac{1}{2}\int_{\Omega}|\Delta u|^2dx-\frac{1}{p+1}\int_{\Omega}|u|^{p+1}dx.\end{equation}
By standard variational methods it is known that \eqref{PP} has a positive least energy solution $u_p$, i.e. the corresponding energy $J_p(u_p)$ is smallest among all nontrivial solutions. The asymptotic behavior of the least energy solutions $u_p$ was studied by Santra-Wei \cite{wei}, where they proved that
\begin{equation}\label{eq-wei}
\lim_{p\to+\infty}p\int_{\Omega}|u_p|^{p+1}dx=64\pi^2e
\end{equation}
and up to a subsequence, \[pu_p(x)\to 64\pi^2\sqrt{e} G(x,x_0) \text{ in } C_{loc}^4(\overline{\Omega}\setminus\{x_0\}).\] Here
$G$ is the Green function of $\Delta^2$ on $\om$ under Dirichlet boundary conditions, that is,
\begin{equation}\label{green}
	 \begin{cases}
\Delta_x^2 G(x,y)=\delta(x-y) \quad &\hbox{in}\;\Omega, \\
G(x,y)=\frac{\partial G(x,y)}{\partial \nu}=0 \ \ &\hbox{for}\;x\in\partial\Omega,
\end{cases}
\end{equation}
and
\[H(x,y):=G(x,y)+\frac{1}{8\pi^2}\log|x-y|\]
 is the regular part of the Green function $G$, and $x_0$ is a critical point of $H(x,x)$. See also \cite{BEG,Tak1,Tak2} for the asymptotic analysis of least energy solutions for the Navier boundary problem
 \[
\begin{cases}
\Delta^2 u =|u|^{p-1}u \quad &\hbox{in}\;\Omega, \\
u=\Delta u=0 \ \ &\hbox{on}\;\partial\Omega.
\end{cases}
\]

On the other hand, comparing to positive solutions, nodal solutions of \eqref{PP} are more difficult to study and not much is known. Here a solution $u$ of \eqref{PP} is called a \emph{nodal} solution, if $u^{\pm}\neq 0$, where $u^+=\max\{u,0\}$ and $u^-=\min\{u,0\}$. A nodal solution is called \emph{a least energy nodal solution} if the corresponding energy $J_p(u)$ is smallest among all nodal solutions.

For the classical Lane-Emden problem
\begin{equation}\label{LE}
\begin{cases}
-\Delta u =|u|^{p-1}u \quad &\hbox{in}\;\Omega\subset\mathbb{R}^2, \\
u=0 \ \ &\hbox{on}\;\partial\Omega,
\end{cases}
\end{equation}
the least energy nodal solutions can be easily obtained by minimizing the functional
\[I(u):=\frac{1}{2}\int_{\Omega}|\nabla u|^2dx-\frac{1}{p+1}\int_{\Omega}|u|^{p+1}dx,\quad u\in H_0^1(\Omega)\]
on the nodal Nehari set
\begin{equation}\label{nehari}\{u\in H_0^1(\Omega) : u^{\pm}\neq 0, I'(u)u^{\pm}=0\}.\end{equation}
However, this approach does not work for the biharmonic problem \eqref{PP}, because
\begin{equation}\text{\it $u\in H_0^2(\Omega)$ can not guarantee $u^{\pm}\in H_0^2(\Omega)$}.\end{equation}
Recently, to overcome this difficulty,  Alves-N\'{o}brega \cite{C.O. Alves2016JDE} applied a dual method to prove the existence of least energy nodal solutions of \eqref{PP} under the following condition on the domain $\Omega$:
\begin{align*}\tag{G} \text{\it The Green function $G$ in $\Omega$ defined by \eqref{green} is positive.}
\end{align*}
It is known that (G) can not hold for arbitrary smooth bounded domains. However, (G) is true for $\Omega$ being a ball or small $C^{4,\alpha}$-smooth perturbations of a ball; see e.g. \cite[Chapter 6]{GGS} and \cite{Grunau2010ARMA positive Green}. Again under the condition (G), the existence of infinitely many nodal solutions \eqref{PP} was proved by Weth \cite{Weth2006dual}.

The purpose of this paper is to study the asymptotic behavior of least energy nodal solutions of \eqref{PP}, the existence of which was given by \cite{C.O. Alves2016JDE}. Remark that in contrast with positive solutions, the asymptotic analysis of nodal solutions is much more difficult and remains largely open. For example, for the Lane-Emden problem \eqref{LE}, the asymptotic behavior of positive solutions has been well studied; see \cite{Francesca2018JMPAMorse index,L-norm2018,Francesca2017positive,Renwei1,Renwei2} and the references therein. However, there are only partial results \cite{Francesca2015JEMS,Grossi2013Poincare,Grunau2014JMPA} concerning the asymptotic behavior of nodal solutions, the study of which is far from complete.

One of our main motivations comes from Grossi-Grumiau-Pacella's seminal work \cite{Grossi2013Poincare}, where they studied the asymptotic behavior of least energy nodal solutions for the Lane-Emden equation \eqref{LE}. Denote $\|u\|_{\infty}:=\sup_{x\in\Omega} |u(x)|$. Among other things, they proved that

\begin{Theorem}\cite{Grossi2013Poincare}\label{thm-LE}
Let $u_p$ be a least energy nodal solution of \eqref{LE}. Then
\be\label{1.4}
p\int_{\om}|\nabla u_p|^2dx\rt16\pi e\quad\quad\text{as }p\rt\iy.
\ee
Furthermore, under the following condition
\begin{align}\tag{B}
    p\left(\Vert u_p^+\Vert_{\infty}-\Vert u_p^-\Vert_{\infty}\right)=O(1)\quad\text{ as }\quad p\rt\iy,
\end{align}
then $\|u_p^{\pm}\|_\infty\to \sqrt{e}$ and
the following statements hold up to a subsequence.
\begin{itemize}
\item[(1)] There are $x^+\neq x^-\in \Omega$ such that
\[pu_p(x)\to 8\pi\sqrt{e}(G(x,x^+)-G(x,x^-))\quad\text{in }C^2_{loc}(\overline{\Omega}\setminus\{x^+,x^-\}).\]
\item[(2)] The points $x^+, x^-$ satisfy
\[
	 \begin{cases}
\nabla_x G(x^+,x^-)=\nabla_x H(x^+,x^+),  \\
\nabla_x G(x^-,x^+)=\nabla_x H(x^-,x^-).
\end{cases}
\]
\end{itemize}
Here $G$ is the Green function of $-\Delta$ in $\Omega\subset\mathbb R^2$ with Dirichlet boundary conditions and $H(x,y)=G(x,y)+\frac1{2\pi}\log|x-y|$.
\end{Theorem}

\begin{remark}
There are counterexamples in \cite{Francesca2015JEMS,Grunau2014JMPA} that the condition (B) fails for some nodal solutions of \eqref{LE}. It was conjectured in \cite{Grossi2013Poincare} that the condition (B) should hold automatically for the least energy nodal solutions of \eqref{LE}. We will study this open question elsewhere.
\end{remark}

In this paper, we generalize those results in \cite{Grossi2013Poincare} \emph{without assuming the condition (B)}. Our main result is
\bt\label{Theorem 1.1} Assume (G) and let $u_p$ be a least energy nodal solution of
 \eqref{PP}. Then as $p\to+\infty$,
\begin{align}\label{new 1,5}
    p\int_{\Omega}|\Delta u_p|^2= p\int_{\om}|u_p|^{p+1}\rightarrow 128\pi^2e,
\end{align}
\begin{align}\label{new 1,5-2}p\int_{\om}|u_p^\pm|^{p+1}\rt64\pi^2e,\quad\|u_p^\pm\|_\infty\to\sqrt{e}.\end{align}
Furthermore, up to a subsequence the following statements hold.
\begin{itemize}
\item[(1)] There are $x^+\neq x^-\in \Omega$ such that
\begin{equation}\label{pup}pu_p(x)\to 64\pi^2\sqrt{e}(G(x,x^+)-G(x,x^-))\quad\text{in }C^4_{loc}(\overline{\Omega}\setminus\{x^+,x^-\}).\end{equation}
\item[(2)] The points $x^+, x^-$ satisfy
\begin{equation}\label{801}
	 \begin{cases}
\nabla_x G(x^+,x^-)=\nabla_x H(x^+,x^+),  \\
\nabla_x G(x^-,x^+)=\nabla_x H(x^-,x^-).
\end{cases}
\end{equation}
\end{itemize}
Here $G$ is the Green function of $\Delta^2$ in $\Omega\subset\mathbb R^4$ defined in \eqref{green}.
\et

\begin{remark}\
\begin{itemize}
\item[(1)] The proof of \eqref{1.4} in Theorem \ref{thm-LE} relies on the technique of the nodal Nehari set \eqref{nehari} and can not work for the biharmonic problem \eqref{PP}. We will develop further the dual method in \cite{C.O. Alves2016JDE}
    to prove \eqref{new 1,5}.

    \item[(2)] Since we do not assume the condition (B), Theorem \ref{Theorem 1.1} can not be proved similarly as Theorem \ref{thm-LE}. Here we need to adopt some ideas from Santra-Wei \cite{wei}, where they studied the asymptotic behavior of solutions $u_p$ of a different problem
        \begin{equation}\label{sw}
\begin{cases}
\Delta^2 u =(u^+)^p \quad &\hbox{in}\;\Omega, \\
u=\frac{\partial u}{\partial \nu}=0 \ \ &\hbox{on}\;\partial\Omega
\end{cases}
\end{equation}
under the energy assumption $\limsup_{p\to\infty}p\int_{\Omega}(u^+)^{p+1}<+\infty$ but without the condition (G). In this case, since the maximum principle fails, $u_p(x)$ can be sign-changing but it was proved in \cite{wei} that $pu_p(x)\geq -C$ for some $C>0$ independent of $p$, so the negative part $pu_p^-(x)$ does not blow up. This is not the case for \eqref{PP}, for which it follows from \eqref{pup} that both $pu_p^+(x)$ and $pu_p^-(x)$ blow up.
\end{itemize}
\end{remark}

For a nodal solution $u_p$ of (\ref{PP}), let $x_p^+$ (resp. $x_p^-$) be a maximum (resp. minimum) point of $u_p$ in $\om$, i.e.
\[u_p(x_p^+)=\Vert u_p^+ \Vert_\infty,\quad
u_p(x_p^-)=-\Vert u_p^- \Vert_\infty.\]
Without loss of generality, we may always assume
\[\Vert u_p\Vert_{\infty}=\Vert u_p^+\Vert_{\infty},\quad\text{i.e.}\quad \Vert u_p^+ \Vert_\infty\geq \Vert u_p^- \Vert_\infty.\]
As in \cite{Grossi2013Poincare,wei}, we consider the scaling
\begin{align}\label{wp+}
    W_{p^\pm}(x):=p\frac{u_p(x_p^\pm+\varepsilon_p^\pm x)-u_p(x_p^\pm)}{u_p(x_p^\pm)},\quad x \in \Omega(\varepsilon_p^\pm):=\frac{\Omega-x_p^\pm}{\varepsilon_p^\pm},
\end{align}
where \[(\varepsilon_p^\pm)^{-4}:=p\Vert u_p^\pm\Vert_{\infty}^{p-1}.\]
We also denote
\begin{equation}\label{omp}\om_p^+:=\{x\in\om:u_p(x)>0\}, \qquad\om_p^-:=\{x\in\om:u_p(x)<0\},\end{equation}
\begin{equation}\label{nsp}\om^\pm(\varepsilon_p^\pm):=\frac{\Omega_p^\pm-x_p^\pm}{\varepsilon_p^\pm},\qquad NS_p:=\{x\in\Omega:u_p(x)=0\}.\end{equation}
Then to prove Theorem \ref{Theorem 1.1}, we need to prove the following result.

\bt\label{th12}
Assume (G) and let $u_p$ be a least energy nodal solution of
 \eqref{PP}. Then up to a subsequence, $W_{p^\pm}(x)$
defined on $\Omega(\varepsilon_p^\pm)$ (resp. defined on $\om^\pm(\varepsilon_p^+)$) converges to $W(x)$ in $C_{loc}^4(\mathbb{R}^4)$,  where $W(x)=-4\log(1+\frac{|x|^2}{8\sqrt{6}})$ solves \begin{equation}\label{P_0}
	 \begin{cases}
	      \Delta^2 W =e^W \quad &\hbox{in}\;\mathbb{R}^4,\\
	      \int_{\mathbb{R}^4}e^W=64\pi^2.
		   \end{cases}
\end{equation}
\et

\begin{remark} \
\begin{itemize}
\item[(1)]
The analogous statement of Theorem \ref{th12} for the Lane-Emden problem \eqref{LE} was proved in \cite{Grossi2013Poincare} under the condition (B). More precisely, under the natural assumption $\Vert u_p^+ \Vert_\infty\geq \Vert u_p^- \Vert_\infty$, the convergence of $W_{p^+}$ was proved in \cite{Grossi2013Poincare} without the condition (B), but in the proof of the convergence of $W_{p^-}$ in \cite{Grossi2013Poincare}, the condition (B) was used to guarantee $\lim_{p\to\infty}\frac{d(x_p^-, NS_p)}{\varepsilon_p^-}=\infty$. Here since we do not assume the condition (B), we need to develop some different techniques from \cite{Grossi2013Poincare}.

\item[(2)]Remark that our proof of Theorem \ref{Theorem 1.1} and Theorem \ref{th12} also works for the Lane-Emden problem \eqref{LE}, and hence gives an alternative proof of those results in \cite{Grossi2013Poincare} without assuming the condition (B). As a consequence, Theorem A actually holds without the condition (B).
\end{itemize}
\end{remark}

The paper is organized as follows. In Section 2, we use the dual method to prove the energy estimate \eqref{new 1,5} for least energy nodal solutions.
In Section 3, we analyze the bubbling phenomenon for general nodal solutions. Finally in Section 4, we turn back to least energy nodal solutions and prove  Theorem \ref{Theorem 1.1} and Theorem \ref{th12}.

{\bf Notations}. Throughout the paper, $C$ always denotes constants that are independent of $p$ (possibly different in different places). Conventionally, we use $o_p(1)=o(1)$ to denote quantities that converge to $0$ as $p\to\infty$. Denote $B_r(x):=\{y\in\R^4 : |y-x|<r\}$. For any $s>1$, we denote $\|u\|_s:=(\int_{\Omega}|u|^s)^{1/s}$.

\section{Energy estimates for least energy nodal solutions}
\label{section-4}

Remark that for general bounded $C^{4,\alpha}$ domain $\Omega\subset\mathbb{R}^4$, the following estimates hold for the Green function $G(x,y)$.

\bl[\cite{A. Dall'Acqua2004estimate green,Grunau2010ARMA positive Green}]\label{Lemma 2.1}
There exists $C>0$ such that for all $x,y\in \Omega,\ x\neq y$, there hold $G(x,y)\geq -C$,
\be\label{2-3}|G(x,y)|\leq C\log\left(1+\frac{1}{|x-y|}\right),\quad |G(x,y)|\leq C(1+|\log|x-y||),\ee
\be\label{2-4} |\nabla^iG(x,y)|\leq\frac{C}{|x-y|^i},\quad \text{for }\ i=1,2,3,4.\ee
\el

In this section, we always assume that the domain $\Omega$ satisfies the condition (G) and prove the energy estimate \eqref{new 1,5} for the least energy nodal solutions.
First we recall the setting of the dual method from
\cite{C.O. Alves2016JDE}. Given any $\omega \in L^{\frac{p+1}{p}}(\Omega)$, it follows from \cite[Chapter 2]{GGS} that the linear problem
\begin{equation}\label{PL}
	 \begin{cases}
\Delta^2 u =\omega \quad &\hbox{in}\;\Omega, \\
u=\frac{\partial u}{\partial \nu}=0 \ \ &\hbox{on}\;\partial\Omega,
\end{cases}
\end{equation}
has a unique solution $u\in W^{4,\frac{p+1}{p}}(\om)\cap H_0^2(\Omega)$, denoted by $u=T\omega=T_p \omega$. So we obtain a linear operator $T=T_p: L^{\frac{p+1}{p}}(\om)\rt W^{4,\frac{p+1}{p}}(\om)\cap H_0^2(\Omega)$ such that $u=T\omega$ is the unique solution of $(\ref{PL})$. Moreover, $T$ satisfies the following properties:
\begin{itemize}
\item[$(T_1)$] $T$ is positive, that is, for any $\omega\in L^{\frac{p+1}{p}}(\Omega)$, $\int_{\Omega}\omega T\omega dx\geq 0$. Furthermore, if $\omega\geq 0$ and $\omega\not\equiv 0$, then $T\omega>0$ in $\Omega$. Clearly this property follows from the condition (G).
\item[$(T_2)$] $T$ is symmetric in the sense that
\[\int_{\Omega}\omega_1 T\omega_2 dx=\int_{\Omega} \omega_2T\omega_1 dx,
\quad\forall \omega_1, \omega_2\in L^{\frac{p+1}{p}}(\Omega).\]
\end{itemize}

Recall that the energy functional of
\begin{equation}\label{P}
	 \begin{cases}
\Delta^2 u =|u|^{p-1}u \quad &\hbox{in}\;\Omega, \\
u=\frac{\partial u}{\partial \nu}=0 \ \ &\hbox{on}\;\partial\Omega,
\end{cases}
\end{equation}
 is given by
\[
    J(u)=J_p(u)=\frac{1}{2}\int_{\om}|\dl u|^2dx-\frac{1}{p+1}\int_{\om}|u|^{p+1}dx,\quad u\in H_0^2(\Omega),
\]
we define the dual functional associated with ($\ref{P}$) by
\begin{align}\label{dual-fun}
   \Psi(\omega)=\Psi_p(\omega):=\frac{p}{p+1}\int_{\om}|\omega|^{\frac{p+1}{p}}dx-\frac{1}{2}\int_{\om}\omega T_p\omega dx,\quad \omega \in L^{\frac{p+1}{p}}(\Omega).
\end{align}
Then $\Psi\in C^1(L^{\frac{p+1}{p}}(\om),\R)$ with
$$\Psi'(\omega)\eta=\int_{\om}|\omega|^{\frac{1-p}{p}}\omega\eta dx-\int_{\om}\eta T_p\omega dx,\quad \forall \omega,\eta\in L^{\frac{p+1}{p}}(\om).$$
Note that
\begin{equation}\label{dual-fun1}\Psi(\omega)=(\frac{p}{p+1}-\frac12)\int_{\om}|\omega|^{\frac{p+1}{p}}dx\quad\text{if }\;
\Psi'(\omega)\omega=0.\end{equation}
We collect the following results
from \cite{C.O. Alves2016JDE}.
\bprop \cite{C.O. Alves2016JDE}\label{new prop4.1}
\begin{itemize}
\item[(i)] $\omega$ is a critical point of $\Psi$ if and only if $u=T\omega$ is a critical point of $J$.  Moreover, $\Psi(\omega)=J(u).$

\item[(ii)]  $u=T\omega$ is a nodal solution of (\ref{P}) if and only if $\omega$ is a nodal critical point of $\Psi$, i.e. $\Psi'(\omega)=0$ and $\omega^{\pm}\neq0$.

\item[(iii)] The functional $\Psi$ has a critical point $\omega_{*} \in L^{\frac{p+1}{P}}(\om)\setminus \{0\}$ such that $\omega_*\geq 0$ and $$\Psi(\omega_{*})=\inf\limits_{u \in \mc N}\Psi(\omega),$$  where
    \begin{equation}\label{mcN}\mc N:=\{\omega\in L^{\frac{p+1}{P}}(\om)\setminus \{0\}\;:\; \Psi'(\omega)\omega=0\}.\end{equation}
    Moreover, letting $u_*=T\omega_*$, we have $u_*>0$ in $\Omega$ and $$J(u_*)=\inf\limits_{u\in \mc N'}J(u)=\Psi(\omega_{*})=\inf\limits_{u \in \mc N}\Psi(\omega),$$ where $$\mc N':=\{u\in H_0^2 (\Omega)\setminus \{0\}\;:\; J'(u)u=0\}.$$
    In other words, $u_*$ is a positive least energy solution of \eqref{P}.

\item[(iv)]  Letting
\begin{equation}\label{mcM}\mc M=\mc M_p:=\{\omega\in L^{\frac{p+1}{p}}(\om)\;:\;\omega^{\pm}\neq0,  \Psi_p'(\omega)\omega^{\pm}=0 \},\end{equation} there is $\omega_0\in\mc M$ such that $$\Psi(\omega_0)=\inf\limits_{\omega\in\mc M}\Psi(\omega),\quad \Psi'(\omega_0)=0. $$ Consequently, $u_0=T\omega_0$ is a least energy nodal solution of ($\ref{P}$), that is
$$J(u_0)=\min\{J(u):u\ \text{is a nodal solution of } (\ref{P})\}=\Psi(\omega_0)=\inf\limits_{\omega\in\mc M}\Psi(\omega).$$
\end{itemize}
\eprop

\begin{remark}\label{rmk}
Since the condition (G) holds for a ball in $\mathbb{R}^4$, we note that properties $(T_1)$-$(T_2)$ and Proposition \ref{new prop4.1} also hold when we replace the domain $\Omega$ with a ball.
\end{remark}

Now we show that the least energy nodal solution of \eqref{P} has the "double energy" property.

\bprop \label{Lem27} Fix any $p>1$ and
recall Proposition \ref{new prop4.1} that $u_*$ is a least energy solution of \eqref{P} and $u_0$ is a least energy nodal solution of \eqref{P}, then $J(u_0)>2J(u_*)$.
\eprop

\begin{proof} Recall Proposition \ref{new prop4.1} that $\omega_0\in\mc M$ and $\Psi(\omega_0)=J(u_0)$. Then
\begin{align}\label{new4.36}
    0=\Psi'(\omega_0)\omega_0^\pm
    =\int_{\om}|\omega_0^\pm|^{\frac{p+1}{p}}-\int_{\om}\omega_0^\pm T\omega_0^\pm-\int_{\Omega}\omega_0^-T\omega_0^+,
\end{align}
where we used $\int_{\Omega}\omega_0^-T\omega_0^+=\int_{\Omega}\omega_0^+T\omega_0^-$ because $T$ is symmetric. Also note that $\int_{\Omega}\omega_0^-T\omega_0^+<0$ because $T$ is positive, i.e. $T\omega_0^+>0$ in $\Omega$. Thus
\begin{equation}\Psi'(\omega_0^\pm)\omega_0^\pm=\int_{\om}|\omega_0^\pm|^{\frac{p+1}{p}}
-\int_{\om}\omega_0^\pm T\omega_0^\pm=\int_{\Omega}\omega_0^-T\omega_0^+<0.\end{equation}
Then there exists $t^\pm\in (0,1)$ such that
\begin{align}\label{new 4.38}
    \Psi'(t^\pm\omega_0^\pm)t^\pm\omega_0^\pm
    =(t^\pm)^{\frac{p+1}{p}}\int_{\om}|\omega_0^\pm|^{\frac{p+1}{p}}-(t^\pm)^2\int_{\om}
   \omega_0^\pm T\omega_0^\pm =0,
\end{align}
i.e. $t^\pm\omega_0^\pm\in \mc N$, where $\mc N$ is given by \eqref{mcN}.
Consequently, it follows from Proposition \ref{new prop4.1} and \eqref{new4.36} that
\begin{align*}
    J(u_*)\leq \Psi(t^\pm\omega_0^\pm)
    =&(\frac{p}{p+1}-\frac{1}{2})(t^+)^{\frac{p+1}{p}}\int_{\om}|\omega_0^\pm|^{\frac{p+1}{p}}
    <(\frac{p}{p+1}-\frac{1}{2})\int_{\om}|\omega_0^\pm|^{\frac{p+1}{p}}\\
    =&\frac{p}{p+1}\int_{\om}|\omega_0^\pm|^{\frac{p+1}{p}}-\frac12\int_{\om}\omega_0^\pm T\omega_0^\pm-\frac12\int_{\Omega}\omega_0^-T\omega_0^+\\
    =&\Psi(\omega_0^\pm)-\frac12\int_{\Omega}\omega_0^-T\omega_0^+.
\end{align*}
From here and \eqref{dual-fun}, we finally obtain
\begin{align*}
J(u_0)=\Psi(\omega_0)
=\Psi(\omega_0^+)+\Psi(\omega_0^-)-\int_{\om}\omega_0^-T\omega_0^+>2J(u_*).
\end{align*}
This completes the proof.
\end{proof}

Now we give the energy estimate of least energy nodal solutions.

\begin{Proposition}\label{Lem26}
Let $(u_p)_{p>1}$ be a family of least energy nodal solutions of \eqref{P}. Then
\begin{align}\label{s2-2-8}
    \lim_{p\to\infty}pJ_p(u_p)=\lim_{p\to\infty}p(\frac{1}{2}-\frac{1}{p+1})\int_{\om}| u_p|^{p+1}= 64\pi^2e.
\end{align}
Consequently, \eqref{new 1,5} holds.
\end{Proposition}

\bp Let $\tilde{u}_p$ be a least energy solution of \eqref{P}. Then as pointed out in \eqref{eq-wei}, it follows from \cite{wei} that
\[\lim_{p\to\infty}pJ_p(\tilde{u}_p)=\lim_{p\to\infty}p(\frac{1}{2}-\frac{1}{p+1})\int_{\om}
|\tilde{u}_p|^{p+1}= 32\pi^2e.\]
Since Proposition \ref{Lem27} proves $J_p(u_p)>2J_p(\tilde{u}_p)$, we obtain
\[
    \liminf_{p\to\infty}pJ_p(u_p)\geq 64\pi^2e.
\]
So it suffices to prove that
\begin{align}\label{s2-2-8-1}
    \limsup_{p\to\infty}pJ_p(u_p)\leq 64\pi^2e.
\end{align}

Take $a_1,a_2\in\om$, $r>0$ such that $B_{2r}(a_1),B_{2r}(a_2)\subset\om$ and $B_{2r}(a_1)\cap B_{2r}(a_2)=\emptyset$.
By Proposition $\ref{new prop4.1}$ and Remark \ref{rmk}, the corresponding equation on the ball $B_r(a_i)$
\begin{equation}
	 \begin{cases}
\Delta^2 u =|u|^{p-1}u \quad &\hbox{in}\;B_r(a_i), \\
u=\frac{\partial u}{\partial \nu}=0 \ \ &\hbox{on}\;\partial B_r(a_i),
\end{cases}
\end{equation}
has a least energy solution $u_{p,i}$, and we may assume
\[\text{ $u_{p,1}(x)>0$ for $x\in B_{r}(a_1)\quad$ and $\quad u_{p,2}(x)<0$ for $x\in B_{r}(a_2)$}.\]
By \cite[Theorem 1.2 and  Corollary 1.3]{wei}, we have that as $p\to\infty$,
\begin{align}\label{new 4.8}
    p\int_{B_r(a_i)}|u_{p,i}|^p\rt64\pi^2\sqrt{e} \quad\text{ and }\quad p\int_{B_r(a_i)}|u_{p,i}|^{p+1}\rt64\pi^2 e,\quad
\end{align}
and for any compact subset $\mathscr{K} \Subset \overline{B}_r(a_i)\setminus \{a_i\}$, there exists a constant $C(\mathscr{K})$ independent of $p$ such that
\begin{align}\label{new 4.9}
    p|\nabla^m u_{p,i}(x)| \leq C(\mathscr{K}), \quad x\in \mathscr{K},\; m=0,1,2,3.
\end{align}

We can consider $u_{p,i}\in L^{p+1}(\Omega)$ by defining $u_{p,i}(x):=0$ for $x\in \Omega\setminus B_{r}(a_i)$.
Then as mentioned before, the linear problem
\begin{equation}\label{2.7}
	 \begin{cases}
\Delta^2 v=\omega_{p,i}:=|u_{p,i}|^{p-1}u_{p,i} \quad &\hbox{in}\;\om ,\\
v=\frac{\partial v}{\partial \nu}=0 \ \ &\hbox{on}\;\partial \om
\end{cases}
\end{equation}
has a unique solution $T_p\omega_{p,i}\in H_0^2(\om)$  such that
\begin{align}
     T_p\omega_{p,1} >0 \,\,\text{ and }\,\, T_p\omega_{p,2} <0\,\, \text{     in } \,\,\Omega.
\end{align}
Recalling the dual functional $\Psi_{p}(\omega)$ in \eqref{dual-fun} and the nodal Nehari set $\mc M_p$ in \eqref{mcM}, since \[\omega_{p,1}\geq 0,\quad \omega_{p,2}\leq 0,\quad\text{supp}\omega_{p,1}\cap \text{supp}\omega_{p,2}=\emptyset,\] it follows from \cite[Lemma 3.3]{C.O. Alves2016JDE}
that there exist $t_p,\ s_p>0$ such that $t_p\omega_{p,1}+s_p\omega_{p,2}\in \mc M_p,$ namely
\begin{align*}
    \Psi_p'(t_p\omega_{p,1}+s_p\omega_{p,2})\omega_{p,1}=0\,\,\text{ and }\,\, \Psi_p'(t_p\omega_{p,1}+s_p\omega_{p,2})\omega_{p,2}=0,
\end{align*}
or equivalently
\begin{equation}\label{2.8-1}
  t_p^{\frac{1}{p}}\int_{\om}|\omega_{p,1}|^{\frac{p+1}{p}}-t_p\int_{\om}\omega_{p,1}T_p\omega_{p,1}=s_p\int_{\om}\omega_{p,1}T_p\omega_{p,2},
\end{equation}
\begin{equation}\label{2.8-2}
  s_p^{\frac{1}{p}}\int_{\om}|\omega_{p,2}|^{\frac{p+1}{p}}-s_p\int_{\om}\omega_{p,2}T_p\omega_{p,2}=t_p\int_{\om}\omega_{p,2}T_p\omega_{p,1}.
\end{equation}
We claim that
\begin{equation}\label{ts}
\limsup_{p\to\infty}t_p\leq 1,\quad\limsup_{p\to\infty}s_p\leq 1.
\end{equation}
Once \eqref{ts} is proved, then since $u_p$ is a least energy nodal solution of \eqref{P}, it follows from Proposition \ref{new prop4.1}-(iv), \eqref{dual-fun1} and \eqref{new 4.8} that
\begin{align*}
  pJ_p(u_p)\leq &p\Psi_p (t_p\omega_{p,1}+s_p\omega_{p,2})= \frac{p(p-1)}{2(p+1)} \int_{\om} |t_p \omega_{p,1}+s_p\omega_{p,2}|^{\frac{p+1}{p}}\\
  =&\frac{p(p-1)}{2(p+1)}t_p^{\frac{p+1}{p}}\int_{B_r(a_1)}|u_{p,1}|^{p+1}+\frac{p(p-1)}{2(p+1)}s_p^{\frac{p+1}{p}}\int_{B_r(a_2)}|u_{p,2}|^{p+1}\\
  \leq &64\pi^2e+o(1),\quad \text{as }p\to\infty,
\end{align*}
namely \eqref{s2-2-8-1} holds.

To prove \eqref{ts}, we need to estimate those integrals in \eqref{2.8-1}-\eqref{2.8-2}.
By the Green representation formula and \eqref{new 4.8}, we have for any $x\in B_r(a_2)$,
\begin{align}
     \left|p T_p \omega_{p,1}(x)\right| =&p \left|\int_{\Omega} G(x,y) \omega_{p,1}(y) dy\right|\leq p\int_{B_r(a_1)} |G(x,y)||u_{p,1}(y)|^p dy\notag\\
     \leq & Cp\int_{B_r(a_1)} |u_{p,1}(y)|^p dy\leq C,
\end{align}
where we used $|G(x,y)|\leq C$ for $x\in B_r(a_2)$ and $y\in B_r(a_1)$ because $B_{2r}(a_1)\cap B_{2r}(a_2)=\emptyset$.
Hence,
\begin{align*}
    p^2\left|\int_{\om}\omega_{p,2}(x)T_p\omega_{p,1}(x)dx\right|\leq&p^2\int_{B_r(a_2)}|u_{p,2}(x)|^p\left|T_p\omega_{p,1}(x)\right|dx\\
    \leq&C p\int_{B_r(a_2)} |u_{p,2}(y)|^p dy \leq C.
\end{align*}
Since $T_p$ is symmetric, we obtain
\begin{align}\label{new4.18}
    \int_{\om}\omega_{p,2}T_p\omega_{p,1}=\int_{\om}\omega_{p,1}T_p\omega_{p,2}=O(\frac{1}{p^2}).
\end{align}

Denote \[v_{p,i}:=pu_{p,i}-p T_p \omega_{p,i},\quad i=1,2.\] It is clear that
\begin{equation}\label{eq-vi}
	 \begin{cases}
\Delta^2 v_{p,i} =0 \quad &\hbox{in}\;B_r(a_i) \\
v_{p,i}=-p T\omega_{p,i},\quad \frac{\partial v_{p,i}}{\partial \nu}=\frac{\partial (-p T\omega_{p,i})}{\partial \nu} \ \ &\hbox{on}\;\partial B_r(a_i).
\end{cases}
\end{equation}

Note from \eqref{2-3} that $|G(x,y)|\leq C(1+|\log|x-y||)$.
Again by the Green representation formula, it follows from \eqref{new 4.8} that for any $x \in \partial B_r(a_i)$,
{\allowdisplaybreaks
\begin{align}
    \left|p T_p \omega_{p,i}(x)\right| =&p \left|\int_{\Omega} G(x,y) \omega_{p,i}(y) dy\right|\leq p\int_{B_r(a_i)}|G(x,y)||u_{p,i}(y)|^{p}dy\notag\\
     \leq &C p\int_{B_r(a_i)\setminus B_{\frac{r}{2}}(a_i)} |\log|x-y|| |u_{p,i}(y)|^p dy\notag\\
     &+C p\int_{B_{\frac{r}{2}}(a_i)} |\log|x-y|| |u_{p,i}(y)|^p dy+C p\int_{B_r(a_i)} |u_{p,i}(y)|^p dy\notag\\
     \leq &C p\int_{B_r(a_i)\setminus B_{\frac{r}{2}}(a_i)} |\log|x-y|| |u_{p,i}(y)|^p dy+C\notag\\
\leq &\frac{C^{p}}{p^{p-1}}\int_{B_{2r}(x)} |\log|x-y||dy+C\leq C,
\end{align}
}%
where we have used $\left\Vert u_{p,i}\right\Vert_{L^{\infty}(B_r(a_i)\setminus B_{r/2}(a_i))}\leq C/p$ by (\ref{new 4.9}).
Thus $p T_p \omega_{p,i}(x)$ is uniformly bounded on $\partial B_r(a_i)$. A similar argument shows that $\frac{\partial (p T_p \omega_{p,i})}{\partial \nu}$ is uniformly bounded on $\partial B_r(a_1)$.
Then by \eqref{eq-vi} and the Green representation formula, we get
\begin{align}
    |v_{p,i}|\leq C \,\,\text{ in }\,\,\ B_r(a_i),\quad i=1,2.
\end{align}
Noting from $\omega_{p,i}=|u_{p,i}|^{p-1}u_{p,i}$ that
\[|\omega_{p,i}|^{\frac{p+1}{p}}-\omega_{p,i}T_p\omega_{p,i}=|u_{p,i}|^{p-1}u_{p,i}(u_{p,i}-
T_p\omega_{p,i})=\frac1p |u_{p,i}|^{p-1}u_{p,i}v_{p,i},\]
it follows that
\begin{align*}
    p^2\left|\int_{\om}\left(|\omega_{p,i}|^{\frac{p+1}{p}}-\omega_{p,i}T_p\omega_{p,i}\right)\right|\leq p\int_{B_r(a_i)}|u_{p,i}|^p|v_{p,i}|\leq C.
\end{align*}
So we have
\begin{align}\label{new 4.25}
    \int_{\om}\left(|\omega_{p,i}|^{\frac{p+1}{p}}-\omega_{p,i}T_p\omega_{p,i}\right)=O(\frac{1}{p^2}),\quad i=1,2.
\end{align}

Now we can apply \eqref{new4.18} and \eqref{new 4.25} to prove the claim \eqref{ts}.  Without loss of generality, we may assume $t_p\geq s_p$ up to a subsequence. First we show that $t_p, s_p$ are uniformly bounded.
By contradiction, we assume that
\begin{align}
    t_p \to +\infty\,\,\text{ and }\,\,\frac{s_p}{t_p} \to C_0\in [0,1].
\end{align}
Then by (\ref{new 4.8}), (\ref{2.8-1}),  (\ref{new4.18}) and (\ref{new 4.25}), we obtain
\begin{align*}
    -64\pi^2 e+o(1)
    =&\left(t_p^{\frac{1}{p}-1}-1\right) p\int_{\om}|u_{p,1}|^{p+1}=\left(t_p^{\frac{1}{p}-1}-1\right) p\int_{\om}|\omega_{p,1}|^{\frac{p+1}{p}}\notag\\
    =& p\int_{\om}\left(\omega_{p,1}T_p\omega_{p,1}-|\omega_{p,1}|^{\frac{p+1}{p}}\right)+\frac{s_p}{t_p} p\int_{\om}\omega_{p,1}T_p\omega_{p,2}\notag\\
    =&O(\frac{1}{p})+(C_0+o(1))O(\frac{1}{p})=o(1),
\end{align*}
which is a contradiction. So $t_p$ and $s_p$ are uniformly bounded.
Assume by contradiction that up to a subsequence,
\[
    t_p \to C_1,\quad s_p \to C_2\quad\text{with }\; \max\{C_1, C_2\}>1,
\]
say $C_1> 1$ for example.
Then by (\ref{new 4.8}), (\ref{2.8-1}) and (\ref{new4.18}), it follows that
\begin{align*}
    (1-C_1)64\pi^2e +o(1)
    = &\left(t_p^{\frac{1}{p}}-t_p\right) p\int_{\om}|u_{p,1}|^{p+1}=\left(t_p^{\frac{1}{p}}-t_p\right) p\int_{\om}|\omega_{p,1}|^{\frac{p+1}{p}}\notag\\
    =& t_p p\int_{\om}\left(\omega_{p,1}T_p\omega_{p,1}-|\omega_{p,1}|^{\frac{p+1}{p}}\right)+s_p p\int_{\om}\omega_{p,1}T_p\omega_{p,2}\notag\\
    = &o(1),
\end{align*}
again a contradiction. This proves the claim \eqref{ts}, so the proof is complete.
\ep

\bl\label{Lemm27} \cite[Lemma 3.4]{wei}
For any $u\in H_0^2(\om)$ and $s>1$, we have $(\int_{\Omega}|u|^s)^{\frac1s}\leq D_ss^{\frac{1}{2}}(\int_{\om}|\dl u|^2)^{\frac{1}{2}},$ where $D_s\rt(64\pi^2e)^{-\frac{1}{2}}$ as $s\rt\iy.$ Remark that the condition (G) is not needed in this result.
\el

\bprop\label{prop4.1}
Let $(u_p)_{p>1}$ be a family of least energy nodal solutions of $(\ref{PP})$. Then
\begin{align}\label{u+}
    p\int_{\om}|u_p^+|^{p+1}\rt64\pi^2e,\quad p\int_{\om}|u_p^-|^{p+1}\rt64\pi^2e.
\end{align}
\eprop
\bp
Denote $\nu_p^+:=|u^+_p|^{p-1}u^+_p$. Then the linear problem
\begin{equation}\label{5.2}
	 \begin{cases}
\Delta^2 v=\nu_p^+ =|u^+_p|^{p-1}u^+_p \quad &\hbox{in}\;\om, \\
v=\frac{\partial v}{\partial \nu}=0 \ \ &\hbox{on}\;\partial \om,
\end{cases}
\end{equation}
has a unique solution $T_p\nu_p^+\in H_0^2(\om)$ which satisfies $T_p\nu_p^+>0$ in $\Omega$.
Multiplying \eqref{5.2} by $T_p\nu_p^+$ and integrating over $\Omega$ leads to
\begin{align*}
  \int_{\om}|\dl T_p\nu_p^+|^2=&\int_{\om} \nu_p^+T_p\nu_p^+\leq\int_{\om}|u_p^+|^p T_p\nu_p^+\notag\\
  \leq&\left(\int_{\om}|u_p^+|^{p+1}\right)^{\frac{p}{p+1}}\left(\int_{\om}(T_p\nu_p^+)^{p+1}\right)^{\frac{1}{p+1}}.
\end{align*}
Similarly, multiplying $u_p$ and using $T_p\nu_p^+>0$ we get
\begin{align*}
    \int_{\om}|u_p^+|^{p+1}=&\int_{\om}u_p\dl^2 T_p\nu_p^+=\int_{\om}T_p\nu_p^+\dl^2 u_p\notag\\
    =&\int_{\om}|u_p|^{p-1}u_pT_p\nu_p^+\leq\int_{\om}|u_p^+|^{p}T_p\nu_p^+\notag\\
    =&\int_{\om}T_p\nu_p^+\dl^2T_p\nu_p^+=\int_{\om}|\dl T_p\nu_p^+|^2.
\end{align*}
From here and Lemma \ref{Lemm27}, we have
\begin{align*}
    64\pi^2e+o(1)=&D_{p+1}^{-2}\leq (p+1)\frac{\int_{\om}|\dl T_p \nu_p^+|^2}{\left(\int_{\om}|T_p\nu_p^+|^{p+1}\right)^\frac{2}{p+1}}\\
    \leq &
    (p+1)\frac{\left(\int_{\om}|\dl T_p \nu_p^+|^2\right)^2}{\left(\int_{\om}|T_p\nu_p^+|^{p+1}\right)^\frac{2}{p+1}\int_{\om}|u_p^+|^{p+1}}\notag\\
    \leq& (p+1)\left(\int_{\om}|u_p^+|^{p+1}\right)^{\frac{p-1}{p+1}}.
\end{align*}
In the same way, we get
$$64\pi^2e+o(1)\leq (p+1)\left(\int_{\om}|u_p^-|^{p+1}\right)^{\frac{p-1}{p+1}}.$$
Since Proposition \ref{Lem26} shows $p\int_{\om}
|u_p|^{p+1}=p\int_{\Omega}|\Delta u_p|^2\to 128\pi^2 e$, we finally obtain \eqref{u+}.
\ep

\section{Blow up analysis of general nodal solutions}
\label{section-3}

In this section, we study the bubbling phenomenon of general nodal solutions (i.e. not just least energy nodal solutions, so we do not need to assume the condition (G) here).
First we recall some known results. The first lemma follows from \cite{Mitidieri1993identity}, and here we use the version from  \cite{wei}.

\bl\label{Pohozaev}\cite{Mitidieri1993identity,wei} (Pohozaev identity)

Let $u\in C^4(\overline{\om})$ be a solution of $\Delta^2u=|u|^{p-1}u$. Then
\begin{align*}
&\frac{4}{p+1}\int_{\Omega}|u|^{p+1}dx\\
=&\int_{\partial\Omega}\lan x-y,\nu\ran \frac{|u|^{p+1}}{p+1}+\frac{1}{2}\lan x-y,\nu \ran (\Delta u)^2-2\frac{\partial u}{\partial \nu} \Delta u dS \\
    -&\int_{\partial\Omega}\frac{\partial (\Delta u)}{\partial \nu}\lan x-y,\nabla u\ran+\frac{\partial u}{\partial \nu}\lan x-y,\nabla (\Delta u)\ran-\lan\nabla u,\nabla (\Delta u)\ran\lan x-y,  \nu\ran dS,
\end{align*}
where $\nu$ is the outer normal vector of $\pa \om$.
\el

The next lemma is the classification result for
\begin{equation}\label{P0}
	 \begin{cases}
	      \Delta^2 W =e^{W} \quad &\hbox{in}\;\mathbb{R}^4,\\
	      \int_{\mathbb{R}^4}e^W<\iy,
		   \end{cases}
\end{equation}
which was obtained by C. S. Lin in \cite{lin1998}.

\bl\cite{lin1998}\label{3} Let $W$ be a solution of \eqref{P0} and
define $$\al:=\frac{1}{32\pi^2}\int_{\R^4}e^{W(x)}dx.$$
Then the following statements hold.
\begin{itemize}
 \item[(1)] After an orthogonal transformation, $W(x)$ can be represented by
\begin{align}\label{21}
  W(x)=&\frac{1}{8\pi^2}\int_{\R^4}\log \frac{|y|}{|x-y|}e^{W(y)}dy-\sum_{j=1}^4a_j(x_j-x_j^0)^2+c_0\notag\\
   =&-\sum_{j=i}^{4}a_j(x_j-x_j^0)^2-4\al\log|x|+c_0+o(1)\quad\text{as} \quad |x|\rt\iy,
\end{align}
where $a_j \geq0$ and $c_0$ are constants, and $x^0=(x_{1}^0,x_{2}^0,x_{3}^0,x_{4}^0)\in \R^4.$ The function $\Delta W$ satisfies
\be\label{22}
\dl W(x)=-\frac{1}{4\pi^2}\int_{\R^4}\frac{e^{W(y)}}{|x-y|^2}dy-2\sum_{j=1}^{4}a_j.
\ee
\item[(2)] The total integration $\alpha\leq 2$.
 If $\al=2$, then all $a_j=0$ and $W$ has the form
\be\label{23}
W(x)=4\log\frac{2\lambda}{1+\lambda^2|x-x^0|^2}+\log24,\quad \lambda>0.
\ee
\item[(3)] If $W(x)=o(|x|^2)$ at $\iy$, then $\al=2.$
\end{itemize}
\el

In the rest of this section, we always assume that
 $(u_p)_{p>1}$ is a family of nodal solutions to (\ref{PP}) satisfying
\begin{equation}\label{3.1}
   p\int_{\om}|\dl u_p|^2dx= p\int_{\om}|u_p|^{p+1}dx \leq C
\end{equation}
for some constant $C$ independent of $p>1$ (Proposition \ref{Lem26} shows that \eqref{3.1} holds automatically for least energy nodal solutions).
Then by H\"{o}lder inequality,
\be\label{est-v-1}
	p\int_\Omega |u_p|^p
	\le \left(p\int_\Omega |u_p|^{p+1}\right)^{\frac{p}{p+1}}(p|\Omega|)^{\frac{1}{p+1}}
	\le C.
\ee
As in Section 1, we let $x_p^{\pm}\in\Omega$ such that
\[u_p(x_p^+)=\Vert u_p^+ \Vert_\infty,\quad
u_p(x_p^-)=-\Vert u_p^- \Vert_\infty,\]
and without loss of generality, we may assume
\[\Vert u_p\Vert_{\infty}=\Vert u_p^+\Vert_{\infty},\quad\text{i.e.}\quad \Vert u_p^+ \Vert_\infty\geq \Vert u_p^- \Vert_\infty.\]
It is known (see e.g. \cite{GGS}) that
$$\lambda_1:=\inf\limits_{\phi\in H_0^2(\om)\setminus\{0\}}\frac{\int_{\om}|\dl\phi(x)|^2}{\int_{\om}|\phi(x)|^2}>0.$$

\begin{Proposition}\label{24}
For any $p>1$, we have
$|u_p(x_p^+)|^{p-1}=\Vert u_p\Vert_{\iy}^{p-1}\geq \lambda_1>0$.
In particular,
$$\liminf_{p\to\infty} \|u_p\|_{\infty}\geq 1,\quad\lim\limits_{p\rt\iy}p\Vert u_p\Vert_{\iy}^{p-1}=+\iy.$$
\end{Proposition}

\begin{proof}
Since $u_p$ satisfies (\ref{PP}), we have
$$\la_1\int_{\om}|u_p|^2\leq \int_{\om}|\dl u_p|^2=\int_{\om}|u_p|^{p+1}\leq u_p(x_p^+)^{p-1}\int_{\om}|u_p|^2,$$
which completes the proof.
\end{proof}

\begin{Proposition}\label{Proposition 2.1}
 There exists $C>0$ independent of $p$ such that $\Vert u_p\Vert_{\iy}\leq C.$
\end{Proposition}

\begin{proof}

By \eqref{2-3}, it was proved in \cite[Proposition 2.7]{Grossi2013Poincare} that there is $C>0$ independent of $x\in\Omega$ such that
\begin{equation}\label{ggg}\|{G(x,\cdot)}\|_{L^{p+1}(\Omega)}^{p+1}\leq C(p+1)^{p+2},\quad \text{for $p$ large enough}.\end{equation}
Then
by the Green representation formula, for any $x\in\Omega$,
$$
	\begin{aligned}
		|u_p(x)|
		&=\Big|\int_\Omega G(x,y)|u_p(y)|^{p-1}u_p(y) dy\Big|\le \|G(x,\cdot)\|_{L^{p+1}(\Omega)}\|{u_p}\|_{L^{p+1}(\Omega)}^p\\
		&\le C^{\frac1{p+1}}(p+1)^{\frac{p+2}{p+1}}\left(\frac{C}{p}\right)^{\frac{p}{p+1}}\le C,\quad\text{for $p$ large enough},
	\end{aligned}
$$
so $\limsup_{p\to\iy}\|{u_p}\|_\iy\le C$.
\end{proof}

Let $W_{p^+}$, $\varepsilon_p^+$,  $\Omega(\varepsilon_p^+)$ and $NS_p$ be defined in \eqref{wp+}-\eqref{nsp}.
It is clear that $W_{p^+}$ satisfies that
\begin{equation}\label{Pp+}
\begin{cases}
\Delta^2 W_{p^+} =\left|1+\frac{W_{p^+}}{p}\right|^{p-1}\left(1+\frac{W_{p^+}}{p}\right) \quad &\hbox{in}\;\Omega(\varepsilon_p^+), \\
W_{p^+}=-p,\frac{\partial W_{p^+}}{\partial \nu}=0 \ &\hbox{on}\;\partial\Omega(\varepsilon_p^+).
\end{cases}
\end{equation}

\bl\label{Lem25}
We have
\[\lim\limits_{p\rt\iy}\frac{d(x_p^+,\pa\om)}{\e_p^+}=\iy.\]
As a result, $\Omega(\varepsilon_p^+)$ converges to $\mathbb{R}^4$ as $p\rt\iy$.
\el

\begin{proof}

The proof can be seen in \cite[Lemma 3.2]{wei}, and we sketch it here for later usage.
Assume by contradiction that up to a subsequence, $\frac{d(x_p^+,\pa\om)}{\e_p^+}\leq C$, then up to a rotation,  $\Omega(\varepsilon_p^{+})=\frac{\Omega-x_p^+}{\varepsilon_p^+}\rt(-l,\iy)\times\R^3$ for some $0\leq l<\iy$. Take $R>\max\{l,1\}$  and let $x\in B_R(0)\cap\Omega(\varepsilon_p^{+})$. Then by the Green representation formula and \eqref{2-4}, we have that for $1\leq i\leq 3$,
\begin{align}\label{as}
  |\nabla^i W_{p^{+}}(x)|&=\frac{p(\e_p^{+})^i}{|u_p(x_p^{+})|}|\nabla^iu_p(x_p^{+}+\e_p^{+}x)| \nonumber\\&=\frac{p(\e_p^{+})^i}{|u_p(x_p^{+})|}\Big |\int_{\om}\nabla_x^iG(x_p^{+}+\e_p^{+}x,y)|u_p(y)|^{p-1}u_p(y)dy\Big|\nonumber\\
  &\leq C\frac{p(\e_p^{+})^i}{|u_p(x_p^{+})|}\int_{B_{2\e_p^{+}R}(x_p^{+})}\frac{1}{|x_p^{+}+\e_p^{+}x-y|^i}|u_p(y)|^pdy\\
  &\quad+C\frac{p(\e_p^{+})^i}{|u_p(x_p^{+})|}\int_{\om\backslash B_{2\e_p^{+}R}(x_p^{+})}\frac{1}{|x_p^{+}+\e_p^{+}x-y|^i}|u_p(y)|^pdy\nonumber\\
  &=I_1+I_2.\nonumber
\end{align}
Note that for $y\in\om\backslash B_{2\e_p^{+}R}(x_p^{+})$, $|x_p^{+}+\e_p^{+}x-y|\geq|x_p^{+}-y|-\e_p^{+}|x|\geq R\e_p^{+}\geq \e_p^{+}$, so we see from Proposition \ref{24} and \eqref{est-v-1} that
\[I_2\le C p\int_{\Omega}|u_p|^pdy\leq C.\]
For $y\in B_{2\e_p^{+}R}(x_p^{+})$, we let $y-x_p^{+}=\e_p^{+}z$ with $z\in B_{2R}(0)$, then it follows from \be\label{2-6}(\e_p^{+})^{-4}=p\|u_p\|_{\infty}^{p-1}=p\|u_p^{+}\|_{\infty}^{p-1}=p|u_p(x_p^{+})|^{p-1},\ee that
\begin{align*}
I_1&\leq C\frac{p(\varepsilon_p^{+})^4|u_p(x_p^+)|^p}{|u_p(x_p^+)|}\int_{B_{2R}(0)}\frac{1}{|x-z|^i}dz
\\
&= C\int_{B_{2R}(0)}\frac{1}{|x-z|^i}dz\leq C\int_{B_{2R}(0)}\frac{1}{|z|^i}dz=CR^{4-i}.
\end{align*}
Therefore, for any $i=1,2,3$,
\be\label{2-9} |\nabla^i W_{p^{+}}(x)|\leq CR^{4-i}+C\leq CR^{4-i},\quad\forall x\in B_R(0)\cap\Omega(\varepsilon_p^{+}),\ee
and then
\be\label{2-9-0}|W_{p^{+}}(x)|=|W_{p^{+}}(x)-W_{p^{+}}(0)|\leq CR^4,\quad\forall x\in \overline{B_R(0)\cap\Omega(\varepsilon_p^{+})}.\end{equation}
From here and $R>l$, we see that there is some point $x_0\in\partial\Omega(\varepsilon_p^{+})$ such that
$p=|W_{p^{+}}(x_0)|\leq C$, a contradiction with $p\to+\infty$.
This proves $\lim\limits_{p\rt\iy}\frac{d(x_p^+,\pa\om)}{\e_p^+}=\infty$.
\end{proof}

\bprop\label{prop22}
Up to a subsequence, $W_{p^+}$ defined on $\Omega(\varepsilon_p^+)$ converges to $W$ in $C_{loc}^4(\mathbb{R}^4)$ as $p\to \infty$, where $W$ solves (\ref{P0}) and $W(x)=-4\log(1+{|x|^2}/{8\sqrt{6}})$.
\eprop

\begin{proof}
The proof can be seen in \cite[Lemma 3.3]{wei}, and we sketch it here for later usage.
By \eqref{2-9}-\eqref{2-9-0}, we have for any fixed $R>1$, there exists $C>0$ such that $|\nabla ^iW_{p^+}(x)| \leq C$ for all $p>1$ and $x\in B_{R}(0),\ i=0,1,2,3$.
Then by standard elliptic estimates, up to a subsequence, there exists $W(x)\leq 0$ such that $ W_{p^{+}}\rt W \text{ in } C_{loc}^4(\R^4)$, and $W$ satisfies $\dl^2W=e^{W}$ in $\mathbb{R}^4$ with $W(0)=0$.
Recall the definition \eqref{wp+} of $W_{p^+}$ that
$1+W_{p^+}(x)/p={u_p(x_p^++\varepsilon_p^+x)}/{u_p(x_p^+)}$,
we see from Fatou Lemma, $(\varepsilon_p^+)^{-4}=pu_p(x_p^+)^{p-1}$, \eqref{est-v-1} and Proposition \ref{24} that
\begin{align*}
    \int_{\R^4}e^{W}dx\leq\liminf_{p\rt\iy} \int_{\Omega(\varepsilon_p^+)}\left|1+\frac{W_{p^+}}{p}\right|^{p}
    =\liminf_{p\rt\iy}
 \frac{p}{u_p(x_p^+)}\int_{\om}\left|u_p\right|^p\leq C.
\end{align*}

Now we show that $ \int_{B_R(0)}|\dl W|dx\leq CR^2.$ Similarly as \eqref{as}, we have
\begin{align*}
    |\dl W_{p^+}(x)|
    \leq C \frac{p(\e_{p}^+)^2}{u_p(x_p^+)}\int_{\om}\frac{|u_p(y)|^{p}}{|x_p^++\e_p^+x-y|^2}dy.
\end{align*}
So for any fixed $R>0$, by Fubini Theorem we have
\begin{align}\label{2.20}
    \int_{B_R(0)}\left|\dl W_{p^+}(x)\right|dx\leq& \frac{Cp}{u_p(x_p^+)}\int_{\om}\left|u_p(y)\right|^p
    \left(\int_{B_R(0)}\frac{(\e_p^+)^2dx}{\left|x_p^++\e_p^+x-y\right|^2}\right)dy.
\end{align}
Changing the variable $z:=x-(y-x_p^+)/\varepsilon_p^+$ gives
\begin{align*}
    \int_{B_R(0)}\frac{(\e_p^+)^2dx}{\left|x_p^++\e_p^+x-y\right|^2}
    =\int_{B_R\left(\frac{x_p^+-y}{\e_p^+}\right)}\frac{dz}{\left|z\right|^2}
    \leq\int_{B_R(0)}\frac{dz}{\left|z\right|^2} \leq CR^2,
\end{align*}
so we see from (\ref{2.20}) and \eqref{est-v-1} that
\begin{align*}
    \int_{B_R(0)}\left|\dl W_{p^+}\right|dx\leq& \frac{CR^2}{u(x_p^+)}p\int_{\Omega} \left|u(y) \right|^p dy\leq CR^2.
\end{align*}
 Letting $p\rt\iy$, it follows that $\int_{B_R(0)}|\dl W|dx\leq CR^2$ for any $R>1$. Then \eqref{22} implies $a_j=0$ for all $j$ and so \eqref{21} gives $W(x)=-4\al\log|x|+c_0+o(1)=o(|x|^2)$ as $|x|\rt\iy.$ Consequently, Lemma \ref{3} implies $\alpha=2$, $W$ is radially symmetric around some point $x^0$ and $\int_{\R^4}e^{W}=64\pi^2.$ Since $W(x)\leq W(0)=0$, we finally get from \eqref{23} that $W(x)=-4\log(1+{|x|^2}/{8\sqrt{6}}).$
\end{proof}

Up to a subsequence, we assume \[\lim_{p\to\infty}x_p^+=x^+\in\overline{\Omega}.\]
Define the blow-up points set of $p u_{p}$ as
\begin{align}
\s=\{x\in\overline{\Omega}:\text{ there exist subsequences }\{p_n\}\ \text{and\ }\{x_n\}\subset\Omega \notag\\
\text{  such that } x_n  \to x \text{ and } |p_n u_{p_n}(x_n)|\to \infty \text{ as } p_n \to \infty\}.
\end{align}
Then $x^+\in \s$. To show that $\s$ is a finite set,
we assume that there exist $n\geq 1$ families of points $(x_{j,p})_{p>1}\in \om$ such that as $p\rt\iy$, $$\e_{j,p}^{-4}:=p|u_p(x_{j,p})|^{p-1}\rt\iy,\quad j=1,2,\ldots,n.$$
Then $$\liminf_{p\to\infty}|u_p(x_{j,p})|\geq1,$$ and so up to a subsequence,
$$\{\lim\limits_{p\rt\iy} x_{j,p} :j=1,2,\ldots,n\} \subset \s  \subset\overline{\om}.$$
We may always take \begin{equation}\label{xp+}x_{1,p}=x_p^+,\quad\text{and so}\quad\e_{1,p}=\e_p^+.\end{equation}
Define \begin{equation}\label{ffc}R_{n,p}(x):=\inf\limits_{j=1,\ldots,n}|x-x_{j,p}|^4,\end{equation}
and we introduce the following properties:
\begin{itemize}
\item[$(\p_1^n)$] For any $j,i\in\{1,2,\ldots,n\},$ $j\neq i,$ $\lim\limits_{p\rt\iy}\frac{|x_{i,p}-x_{j,p}|}{\e_{j,p}}=\iy$.
\item[$(\p_2^n)$] For any $j\in\{1,2,\ldots,n\}$, after passing to a subsequence, $$W_{j,p}(x):=\frac{p}{u_p(x_{j,p})}\left(u_p(x_{j,p}+\e_{j,p}x)-u_p(x_{j,p})\right)\rt W(x)$$ in $C_{loc}^4(\R^4)$ as $p\rt\iy,$ where $W(x)=-4\log(1+\frac{|x|^2}{8\sqrt{6}})$.\\
\item[$(\p_3^n)$] There exists $C>0$ such that $pR_{n,p}(x)|u_p(x)|^{p-1}\leq C$ for all $p>1$ and $x\in\om.$
\end{itemize}
Proposition \ref{prop22} and \eqref{xp+} shows that $(\p_1^1)$-$(\p_2^1)$ hold. Note also that if $(\p_3^n)$ holds for some $n\geq 1$, then we can not find the $(n+1)$-th family of points $(x_{n+1,p})_{p>1}\in \om$ such that $\e_{n+1,p}^{-4}:=p|u_p(x_{n+1,p})|^{p-1}\rt\iy$ and $(\p_1^{n+1})$ holds, i.e. $(\p_1^{n+1})$-$(\p_2^{n+1})$ can not hold.


\begin{Proposition}\label{prop31}
There exists $k\in \N_{\geq 1}$ and $k$ families of points $(x_{j,p})_{p>1}$ in $\om,$ $j=1,\ldots,k$ such that, after passing to a sequence, the following statements hold.
\begin{itemize}
\item[(1)] $(\p_1^k),\ (\p_2^k)$ and $(\p_3^k)$ hold. Consequently, we can not find the $(k+1)$-th family of points $(x_{k+1,p})_{p>1}\in \om$ such that $(\p_1^{k+1})$-$ (\p_2^{k+1})$ hold.
\item[(2)] $$\s=\{\lim\limits_{p\rt\iy} x_{j,p} :j=1,2,\ldots,k\},$$
so $\s$ is finite.
Furthermore, given any compact subset $\mathscr{K}\Subset\overline{\Omega}\setminus\s$, there exists a constant $C(\mathscr{K})>0$ independent of $p$ such that
\begin{align}\label{31}
  p \left\|\nabla^m u_p\right\|_{L^\infty(\mathscr{K})}\leq C(\mathscr{K}) ,\quad  m=0,1,2,3.
\end{align}
\end{itemize}
\end{Proposition}

\begin{proof} The proof is almost the same as that of \cite[Lemma 4.1]{wei} with trivial modifications, and we omit it here.
\end{proof}

Recall \eqref{nsp} that $NS_p:=\{x\in\Omega:u_p(x)=0\}$.
\begin{corollary}\label{cor31}
We have
$$\lim_{p\to\infty}\frac{d(x_{j,p},NS_{p})}{\e_{j,p}}=\iy,\quad \text{ for  all}\ j=1,2,\cdots,k.$$
Consequently, by setting $u_{j,p}:=u_{p}\chi_{\om_{j,p}}$, where $\om_{j,p}$ is the nodal domain of $u_{p}$ containing $x_{j,p}$ and $\chi_{\om_{j,p}}$ denotes the characteristic function of $\om_{j,p}$, then
$$\widetilde{W}_{j,p}:=\frac{p}{u_{p}(x_{j,p})}(u_{j,p}(x_{j,p}+\e_{j,p}x)-u_{p}(x_{j,p})),$$
defined on $\frac{\om_{j,p}-x_{j,p}}{\e_{j,p}}$ converges to $W=-4\log(1+{|x|^2}/{8\sqrt{6}})$ in $C^4_{loc}(\mathbb{R}^4)$ as $p\rt\iy$.
\end{corollary}

\bp
Assume by contradiction that up to a subsequence, for some $j$,
\begin{align}
    \lim\limits_{p\rt\iy}\frac{d(x_{j,p},NS_{p})}{\e_{j,p}}=M<\iy.
\end{align}
Then up to a subsequence, there exists $y_{j,p}\in NS_{p}$ such that as $p\to \infty$, \begin{align}
    \frac{y_{j,p}-x_{j,p}}{\e_{j,p}}\rt x_{\iy}\quad  \text{ for some } |x_{\iy}|=M,
\end{align}
so by the convergence of $W_{j,p}$ in $(\p_2^{k})$, we have
\begin{align}
-p=W_{j,p} \Big(\frac{y_{j,p}-x_{j,p}}{\e_{j,p}}\Big)\rt W(x_{\iy}),
\end{align}
which is a contradiction.
\ep

\begin{remark}\label{rmk3-9} Recalling Proposition \ref{prop31}, we write
\[\s=\{x_1,\ldots,x_N\}\quad\text{with }\quad x_1=x^+=\lim_{p\to\infty}x_p^+.\]
Since there may happen $\lim\limits_{p\rt\iy} x_{i,p}=\lim\limits_{p\rt\iy} x_{j,p}$ for some $i\neq j\in \{1,\cdots,k\}$, we have $1\leq N\leq k$. Take $r_0>0$ small such that
\begin{align*}
    B_{2r_0}(x_i) \cap B_{2r_0}(x_j)=\emptyset, \quad \forall x_i\neq x_j\in S.
\end{align*}
\end{remark}

As in Section 1, we denote
\begin{equation}\label{omp1}\om_p^+:=\{x\in\om:u_p(x)>0\}, \qquad\om_p^-:=\{x\in\om:u_p(x)<0\}.\end{equation}

\begin{Proposition}\label{prop3.2}
Up to a subsequence, $$pu_{p}(x) \rt\sum_{i=1}^{N}\gamma_i G(x,x_i)\ \ \text{in }C_{loc}^4(\overline{\om}\backslash\s),$$
where\begin{align}
  \gamma_i:=&\lim\limits_{\rho\rt 0} \lim\limits_{p \to \iy} p\int_{B_{\rho }(x_i)}|u_{p}(x)|^{p-1} u_{p}(x)dx=\gamma_i^+-\gamma_i^-, \label{38}\\
  \gamma_i^\pm:=&\lim\limits_{\rho\rt 0}\lim\limits_{p \to \iy} p\int_{B_{\rho }(x_i)\cap \Omega_p^\pm}|u_{p}(x)|^{p}dx.\notag
\end{align}
\end{Proposition}

\begin{proof} The proof is standard. Let $\mathscr{K}\Subset\overline{\om}\backslash\s$ be any compact subset and take $r\in (0,r_0)$ small such that
\begin{align}\label{new3.7}
    \mathscr{K}\subset\overline{\Omega}\backslash\left(\bigcup\limits_{i=1}^N B_{2r}(x_i)\right).
\end{align} Then for $m\in\{0,1,2,3\}$ and $x\in\mathscr{K}$,
it follows from the Green representation formula that
\begin{align}\label{3-22}
    p\nabla^m u_{p}(x)=&\int_{\om}\nabla^m_x G(x,y)p|u_{p}(y)|^{p-1}u_{p}(y)dy\notag\\
    =&\int_{\om_{p}^+} \nabla^m_x G(x,y)p|u_{p}(y)|^{p}dy-\int_{\om_{p}^-}\nabla^m_x G(x,y)p|u_{p}(y)|^{p}dy.
\end{align}
For any $\rho\in (0,r)$, we have
\begin{align}\label{ccc}
    &\int_{\om_{p}^+}\nabla^m_x G(x,y)p|u_{p}(y)|^{p}dy\notag\\ =&\sum_{i=1}^{N}\nabla^m_x G(x,x_i)\int_{\Omega_p^+\cap B_\rho(x_i)}p|u_{p}(y)|^{p}dy\\&+
    \sum_{i=1}^{N}\int_{\Omega_p^+\cap B_\rho(x_i)}\nabla^m_x \left(G(x,y)- G(x,x_i)\right)p|u_{p}(y)|^{p}dy\notag\\
    &+\int_{\om_{p}^+\backslash \bigcup_{i=1}^{N} B_\rho(x_i)} \nabla^m_x G(x,y) p|u_{p}(y)|^{p}dy.\notag
\end{align}
Since (\ref{31}) implies
\begin{align}\label{new3.8}
    \Vert u_{p}\Vert_{L^{\iy}(\om \backslash \bigcup_{i=1}^{N} B_{\rho}(x_i))}\leq \frac{C(\rho)}{p},
\end{align}
and Lemma \ref{Lemma 2.1} implies
\[
    \int_{\om_{p}^+\backslash \bigcup_{i=1}^{N} B_\rho(x_i)} |\nabla^m_x G(x,y)|dy\leq \int_{\om} |\nabla^m_x G(x,y)|dy \leq C,
\]
we obtain
\begin{align}\label{995}
   \left| \int_{\om_{p}^+\backslash \bigcup_{i=1}^{N} B_\rho(x_i)} \nabla^m_x G(x,y) p|u_{p}(y)|^{p}dy\right|\leq C\frac{(C(\rho))^{p}}{p^{p-1}}.
\end{align}
On the other hand,  by (\ref{new3.7}) and Lemma \ref{Lemma 2.1} we have
\begin{align}\label{new3.11}
    |\nabla_y\nabla_x^mG(x,y)|\leq C,\quad\forall x\in\mathscr{K},\; y\in \cup_{i=1}^N B_{r}(x_i),
\end{align}
so
\begin{align}
    &\left|\int_{\Omega_p^+\cap B_\rho(x_i)}\nabla^m_x \left(G(x,y)- G(x,x_i)\right)p|u_{p}(y)|^{p}dy\right|\notag\\
    \leq&\int_{\Omega_p^+\cap B_\rho(x_i)}\rho|\nabla_y\nabla_x^mG(x,\lambda y+(1-\lambda)x_i)|p|u_{p}(y)|^{p}dy\leq C\rho,\label{new3.12}
\end{align}
where $\lambda\in (0,1)$ depends on $y$.
Inserting (\ref{995}) and (\ref{new3.12}) into \eqref{ccc} we get
\begin{align}
   \lim\limits_{p\to \infty} \int_{\om_{p}^+}\nabla^m_x G(x,y)p|u_{p}(y)|^{p}dy=\sum_{i=1}^{N}\gamma_i^+ G(x,x_i).
\end{align}
Similarly, we can prove \[\lim\limits_{p\to \infty} \int_{\om_{p}^-}\nabla^m_x G(x,y)p|u_{p}(y)|^{p}dy=\sum_{i=1}^{N}\gamma_i^- G(x,x_i).\]
From here and \eqref{3-22}, we conclude that $pu_{p}(x) \rt\sum_{i=1}^{N}\gamma_i G(x,x_i)$ in $C_{loc}^3(\overline{\om}\backslash\s)$. Then by standard elliptic estimates, the convergence is also in $C_{loc}^4(\overline{\om}\backslash\s)$.
\end{proof}

\begin{remark}
A challenging open problem is to determine the possible values of $\gamma_i$ for general nodal solutions. In the next section, we will prove $\gamma_i=\pm 64\pi^2\sqrt{e}$ for least energy nodal solutions.
\end{remark}

Next, we will prove that no boundary blow-up occurs.
\begin{Proposition}\label{prop 33}
$\s\cap\pa\om=\emptyset$.
\end{Proposition}
\begin{proof}
The proof is similar to that of \cite[Lemma 4.3]{wei}, and we sketch it here for the reader's convenience.
Assume by contradiction that $\s\cap\pa\om\neq\emptyset$, say $x_1\in \s\cap\pa\om $ for example. Then Remark \ref{rmk3-9} indicates that $B_{2r_0}(x_1)\cap \s=\{x_1\}$. By Proposition \ref{prop31}, we take $1\leq j\leq k$ such that $x_1=\lim_{p\to\infty} x_{j,p}$. Then by Proposition \ref{prop31} and $\liminf_{p\to\infty}|u_p(x_{j,p})|\geq1$, we have that for any $r\in(0,r_0)$,
\begin{align}\label{765}
    p\int_{B_{r}(x_1)\cap \Omega}\left|u_p(x)\right|^{p+1} dx \geq &
    p\int_{B_{\frac{r}{2}}(x_{j,p})\cap \Omega}\left|u_p(x)\right|^{p+1}  dx\notag\\
    =& \left|u_p(x_{j,p})\right|^2\int_{B_{\frac{r}{2\varepsilon_{j,p}}}(0)\cap \frac{\Omega-x_{j,p}}{\varepsilon_{j,p}}}\left|1+\frac{W_{j,p}(z)}{p}\right|^{p+1}  dz\notag\\
    \geq& \left|u_p(x_{j,p})\right|^2(\int_{\R^4}e^W+o(1))\geq 64\pi^2+o(1),\;\text{as $p\to\infty$.}
\end{align}

Let $\nu$ be the outer normal vector of $(\partial\Omega \cap B_r(x_1)) \cup (\Omega\cap\partial B_r(x_1))$, and define $y_{p}:=x_1+\rho_{r,p}\nu (x_1)$, where
\begin{align}
    \rho_{r,p}:=\frac{\int_{\partial\Omega \cap B_r(x_1)} \langle x-x_1 , \nu \rangle \left| \Delta u_{p}(x) \right|^2 dS}{\int_{\partial\Omega \cap B_r(x_1)} \langle \nu(x_1) , \nu \rangle \left| \Delta u_{p}(x) \right|^2 dS}.
\end{align}
Hence,
\begin{align}\label{3-31}
    \int_{\partial\Omega \cap B_r(x_1)} \langle x-y_{p} , \nu \rangle \left| \Delta u_{p}(x) \right|^2 dS =0.
\end{align}
We can choose $r>0$ small enough such that $\frac{1}{2}\leq \langle \nu(x_1) , \nu \rangle \leq 1$ for $x\in \partial\om\cap B_r(x_1)$. Then $|\rho_{r,p}|\leq 2r$ . Using Lemma \ref{Pohozaev} (i.e. the Pohozaev identity) on $\Omega \cap B_r(x_1)$ with
$y=y_{p}$, $u_p=\nabla u_p=0$ on $\partial\Omega$ and \eqref{3-31}, we obtain
{\allowdisplaybreaks
\begin{align}
     &\frac{4{p}^2}{p+1}\int_{\Omega \cap B_r(x_1)}|u_{p}|^{{p}+1} dx\notag\\
   =&\frac{{p}^2}{p+1}\int_{\Omega \cap \partial B_r(x_1)}\langle x-y_{p}, \nu \rangle|u_{p}|^{{p}+1}dS
 -2\int_{\Omega \cap \partial B_r(x_1)}\frac{\partial(pu_p)}{\partial\nu}\dl ({p} u_{p}) dS\notag\\
  &+\frac{1}{2}\int_{\Omega \cap \partial B_r(x_1)}\langle x-y_{p}, \nu\rangle( \dl ({p}u_{p}))^2dS-\int_{\Omega \cap \partial B_r(x_1)}\langle x-y_{p}, \nabla({p}u_{p})\rangle\frac{\pa \dl (pu_{p})}{\partial \nu}dS\notag\\
  &-\int_{\Omega \cap \partial B_r(x_1)}\langle x-y_{p}, \nabla\dl({p}u_{p})\rangle\frac{\pa (pu_{p})}{\partial \nu} dS\notag\\
  &+\int_{\Omega \cap \partial B_r(x_1)}\langle x-y_{p},\nu \rangle\langle\nabla({p}u_{p}),\nabla\dl({p}u_{p})\rangle dS\notag\\  =&\RNum{1}_p+\RNum{2}_p+\RNum{3}_p+\RNum{4}_p+\RNum{5}_p+\RNum{6}_p.\notag
\end{align}
}%
Recall Proposition $\ref{prop3.2}$ that $pu_{p}(x)\rt \sum_{i=1}^{N} \gamma_i G(x,x_i)=\sum_{i=2}^{N} \gamma_i G(x,x_i)$ \text{in} $C_{loc}^4(\overline{\om}\backslash\s)$, where $G(x,x_1)=0$ is used since $x_1\in\partial\Omega$.
From here and $|x-y_{p}|\leq |x-x_1|+|\rho_{r,p}|\leq 3r$ for $x \in \Omega \cap \partial B_r(x_1)$, we easily obtain the following estimates
\[\RNum{2}_p=O(r^3)\quad\text{and }\; \RNum{3}_p,\; \RNum{4}_p,\; \RNum{5}_p, \; \RNum{6}_p=O(r^4).\]
On the other hand, it follows from (\ref{31}) that
\begin{align*}
    |\RNum{1}_p| \leq \frac{{p}^2}{p+1}\Big(\frac{C}{p}\Big)^{p+1} \int_{\Omega \cap \partial B_r(x_1)}
    |x-y_{p}| dS =
    o_{p}(1) r^4.
\end{align*}
Inserting these estimates into the above identity leads to
\begin{align}
    \lim \limits_{r\to 0}\lim \limits_{p \to \infty}\frac{4{p}^2}{p+1}\int_{\Omega \cap B_r(x_1)}|u_{p}|^{{p}+1} dx=0,
\end{align}
which contradicts with $(\ref{765})$.
\end{proof}

Thanks to Proposition \ref{prop 33}, in Remark \ref{rmk3-9} we can further assume $B_{2r_0}(x_i)\subset\Omega$ for all $i$ by taking $r_0$ smaller if necessary.

\begin{Lemma}\label{local} Fix any $x_i\in \s$ and recall $\gamma_i=\gamma_i^+-\gamma_i^-$ in Proposition \ref{prop3.2}. Then
\begin{align}\label{3-34}
  \lim\limits_{\rho \to 0}\lim\limits_{p\to \infty} p\int_{B_{\rho}(x_i)}|u_{p}|^{{p}+1} dx=\frac{\gamma_i^2}{64 \pi^2}=\frac{(\gamma_i^+-\gamma_i^-)^2}{64 \pi^2}.
\end{align}
Consequently,
\begin{align}\label{3-34-2}
\lim\limits_{p\to \infty} p\int_{\Omega}|u_{p}|^{{p}+1} dx=\frac{1}{64 \pi^2}\sum_{j=1}^N{\gamma_j^2}=\frac{1}{64 \pi^2}\sum_{j=1}^N(\gamma_j^+-\gamma_j^-)^2.
\end{align}
\end{Lemma}

\begin{proof}
By Proposition \ref{prop3.2}, we have
\begin{align}
    pu_{p}(x) \rt\sum_{j=1}^N \gamma_jG(x,x_j) \quad\text{in }C_{loc}^4 (\overline{B_{r_0}(x_i)} \backslash \{ x_i \}).
\end{align}
Then for any fixed $\rho\in (0,r_0)$, it follows that for $x\in \overline{ B_{r_0}(x_i)}\setminus B_{\frac{\rho}{2}}(x_i)$,
{\allowdisplaybreaks
\begin{align}
   & p u_{p}(x)= -\frac{\gamma_i}{8\pi^2}\log{\abs{x-x_i}}+O(1),\label{U1}\\
   & p \nabla u_{p }(x) = -\frac{\gamma_i}{8\pi^2}\frac{x-x_i}{\abs{x-x_i}^2}+O(1),\nonumber\\
   & p \Delta u_{p}(x)= - \frac{\gamma_i}{4\pi^2}\frac{1}{\abs{x-x_i}^2}+O(1),\nonumber\\
   & p \nabla \Delta u_{p}(x)= \frac{\gamma_i}{2\pi^2}\frac{x-x_i}{\abs{x-x_i}^4}+O(1).\label{U4}
\end{align}
}%
On the other hand,
by using Lemma $\ref{Pohozaev}$ (i.e. the Pohozaev identity) on $B_{\rho}(x_i)$ with
$y=x_i$ and noting $ x-x_i=\rho\nu$ for $x\in\partial B_{\rho}(x_i)$, we get
\begin{align}\label{3-39}
     &\frac{4{p}^2}{p+1}\int_{B_{\rho}(x_i)}|u_{p}|^{{p}+1} dx\notag\\
   =&\rho\frac{{p}^2}{p+1}\int_{\partial B_{\rho}(x_i)}|u_{p}|^{{p}+1}dS
  -2\int_{\partial B_{\rho}(x_i)}\langle \nabla({p}u_{p}) ,\nu \rangle\dl ({p} u_{p}) dS\notag\\
  +&\frac{\rho}{2}\int_{\partial B_{\rho}(x_i)}( \dl ({p}u_{p}))^2dS- 2\rho\int_{\partial B_{\rho}(x_i)}\langle \nabla({p}u_{p}), \nu\rangle
  \langle \nabla \Delta({p}u_{p}) ,\nu \rangle dS\notag\\
  +&\rho\int_{\partial B_{\rho}(x_i)}\langle\nabla({p}u_{p}),\nabla\dl({p}u_{p})\rangle dS.
\end{align}
From (\ref{U1})-(\ref{U4}), we can obtain that
\[
   \rho\frac{{p}^2}{p+1}\int_{\partial B_{\rho}(x_i)}|u_{p}|^{{p}+1}dS
    \leq C\frac{p^2}{p+1}\Big(\frac{C}{p}\Big)^{p+1}=o_p(1),
\]
\[
    -2\int_{\partial B_{\rho}(x_i)}\langle \nabla({p}u_{p}) ,\nu \rangle\dl ({p} u_{p}) dS=-\frac{\gamma_i^2}{8\pi^2}+ O(\rho),
\]
\[
    \frac{\rho}{2}\int_{\partial B_{\rho}(x_i)}( \dl ({p}u_{p}))^2dS=\frac{\gamma_i^2}{16\pi^2}+O(\rho),
\]
\[
    -2\rho\int_{\partial B_{\rho}(x_i)}\langle\nabla({p}u_{p}),\nu\rangle
    \langle \nabla \Delta({p}u_{p}) ,\nu \rangle dS=\frac{\gamma_i^2}{4\pi^2}+O(\rho),
\]
\[
    \rho\int_{\partial B_{\rho}(x_i)}\langle\nabla({p}u_{p}),\nabla\dl({p}u_{p})\rangle dS=-\frac{\gamma_i^2}{8\pi^2}+O(\rho).
\]
Inserting these estimates into \eqref{3-39}, we obtain \eqref{3-34}.

Consequently, since \eqref{31} implies
\begin{align}\label{enen}p\int_{\Omega}|u_p|^{p+1}&=\sum_{j=1}^Np\int_{B_\rho(x_j)}|u_p|^{p+1}
+\int_{\Omega\setminus \cup_jB_\rho(x_j)}|u_p|^{p+1}\\
&=\sum_{j=1}^Np\int_{B_\rho(x_j)}|u_p|^{p+1}+o_p(1),\nonumber\end{align}
we obtain \eqref{3-34-2}.
\end{proof}

The following lemma is important when we consider least energy nodal solutions in the next section. Roughly speaking, this lemma asserts that the sequences $(x_{j,p})_{p>1}$ in Proposition \ref{prop31} can be replaced with sequences consisting of local maximum points of $|u_p(x)|$. Recall $\Omega_p^{\pm}$ defined in \eqref{omp1}.

\begin{Lemma}\label{Lem31}
Fix any $x_i \in \s$. Given any $r\in(0,r_0)$, we take ${y_{i,p}}\in\overline{B_{r}(x_i)}$ such that
\begin{align}\notag
    |u_{p}(y_{i,{p}})|
:=\max\limits_{\overline{B_{r}(x_i)}}|u_{p}(x)|.
\end{align}
Then up to a subsequence, if $y_{i,p}\in \Omega_p^+$ (the case $y_{i,p}\in \Omega_p^-$ can be treat similarly and is omitted here), we have \begin{itemize}
\item[(1)] $\e_{i,p}^{*}:=[p|u_{p}(y_{i,p})|^{p-1}]^{-\frac{1}{4}}\rt0$ and ${y_{i,p}}\rt x_{i}.$
\item[(2)] Define \begin{equation}\label{www}\widetilde{W}_{i,p}^{*}:=\frac{p}{u_{p}(y_{i,p})}(u_{p}(y_{i,p}+\e_{i,p}^{*}x)-u_{p}(y_{i,p})),\quad x\in\frac{B_{r}(x_i)\cap \Omega_p^+-y_{i,p}}{\e_{i,p}^{*}},\end{equation}
then
$\widetilde{W}_{i,p}^{*}\rt W=-4\log(1+\frac{|x|^2}{8\sqrt{6}})$ in  $C_{loc}^4(\R^4).$
\item[(3)]
\begin{align}\label{3-35}
\liminf\limits_{p \to \iy}p\int_{B_{r}(x_i)\cap \Omega_p^+}|u_{p}(x)|^{p+1}dx\geq 64\pi^2\liminf\limits_{p \to \iy}|u_{p}(y_{i,p})|^2\geq 64\pi^2,
\end{align}
\begin{align}\label{3-36}
    \liminf\limits_{p \to \iy}p\int_{B_{r}(x_i)\cap \Omega_p^+}|u_{p}(x)|^{p}dx\geq 64\pi^2\liminf\limits_{p \to \iy}|u_{p}(y_{i,p})|\geq 64\pi^2.
\end{align}
\end{itemize}
\end{Lemma}

\begin{proof}
By ${y_{i,p}}\in\overline{B_{r}(x_i)}$ and the definition of $R_{k,p}(y_{i,p})$ in \eqref{ffc}, there exists $1\leq j\leq k$ such that
\begin{align}\label{new 3.22}
    R_{k,p}(y_{i,p})= |x_{j,p}-y_{i,p}|^4 \quad\text{  and  }\quad \lim\limits_{p\to \infty} x_{j,p}=x_i.
\end{align}

(1). By the definition of $y_{i,p}$, we have $|u_p(y_{i,p})| \geq |u_p(x_{j,p})|>0$, so
\begin{align}\label{3-37}
    0<\e_{i,p}^{*}:=[p|u_{p}(y_{i,p})|]^{-\frac{1}{4}}\leq [p|u_{p}(x_{j,p})|]^{-\frac{1}{4}}=\e_{j,p}\rt0.
\end{align}
Furthermore, it follows from $(\p_3^{k})$ that
\begin{align}\label{new3.24}
    \frac{|x_{j,p}-y_{i,p}|}{\e_{i,p}^{*}}=[pR_{k,p}(y_{i,p})|u_p(y_{i,p})|^{p-1}]^{\frac{1}{4}}\leq C,
\end{align}
so $\lim_{p\to\infty}{y_{i,p}}=\lim_{p\to\infty}{x_{j,p}}=x_{i}$.

(2). Recall \eqref{3-37} that $\e_{i,p}^{*}/\e_{j,p}\leq 1$.
Then by (\ref{new3.24}) we get
\begin{align}
    \frac{|x_{j,p}-y_{i,p}|}{\e_{j,p}}=\frac{|x_{j,p}-y_{i,p}|}{\e_{i,p}^{*}}
    \frac{\e_{i,p}^*}{\e_{j,p}} \leq C.\notag
\end{align}
So it follows from $(\p_2^{k})$ that there is $x_{\infty} \in \mathbb{R}^4$ with $|x_{\infty}|\leq C$ such that up to a subsequence,
\begin{align*}
    W_{j,p}\left( \frac{y_{i,p}-x_{j,p}}{\e_{j,p}}\right) \to W(x_{\infty})\leq 0.
\end{align*}
As a consequence,
\begin{align}
    1\leq\left(\frac{\e_{j,p}}{\e_{i,p}^*}\right)^4=\left|\frac{u_{p}(y_{i,p})}{u_{p}(x_{j,p})}\right|^{p-1}=\left|1+\frac{W_{j,p}\left( \frac{y_{i,p}-x_{j,p}}{\e_{j,p}}\right)}{p}\right|^{p-1} \to e^{W(x_{\infty})}\leq 1,
\end{align}
which implies
\begin{align}\label{new q q}
    \frac{\e_{j,p}}{\e_{i,p}^*} \to 1 .
\end{align}

Recall $y_{i,p}\in\Omega_p^+$. If $u_p(x_{j,p})<0$ up to a subsequence, then by  $\e_{i,p}^{*}/\e_{j,p}\leq 1$ and Corollary \ref{cor31},
\begin{align*}
    \frac{|x_{j,p}-y_{i,p}|}{\e_{i,p}^*}\geq\frac{d(x_{j,p},NS_{p})}{\e_{j,p}} \to \infty,
\end{align*}
which contradicts with $(\ref{new3.24})$. Thus $u_p(x_{j,p})>0$, i.e. $(x_{j,p})_{p>1}\in \Omega_p^+$ for $p$ large.
Consequently,
\begin{align*}
    \frac{d(y_{i,p},NS_{p})}{\e_{i,p}^*}\geq \frac{d(x_{j,p},NS_{p})}{\e_{j,p}}-\frac{|x_{j,p}-y_{i,p}|}{\e_{i,p}^*}\geq \frac{d(x_{j,p},NS_{p})}{\e_{j,p}} -C \rt\iy,
\end{align*}
and $y_{i,p}\to x_i$ yields
\begin{align*}
   \frac{d(y_{i,p},\partial B_{r}(x_i))}{\e_{i,p}^*} \geq\frac{r}{2\e_{i,p}^{*}}   \rt\iy.
\end{align*}
From here we get
\begin{align}\label{3-45}
    \frac{B_{r}(x_i)\cap \Omega_p^+-y_{i,p}}{\e_{i,p}^{*}}\rt\R^4\quad \text{as } p \to \infty.
\end{align}

Recalling the $\widetilde{W}_{i,p}^{*}(x)$ defined in \eqref{www}, we have
\be
\begin{cases}\notag
  \dl^2\widetilde{W}_{i,p}^{*}(x)=\left | 1+\frac{\widetilde{W}_{i,p}^{*}(x)}{p}\right | ^{p-1}\left(1+\frac{\widetilde{W}_{i,p}^{*}(x)}{p}\right),\quad x\in \frac{B_{r}(x_i)\cap \Omega_p^+-y_{i,p}}{\e_{i,p}^{*}},\\
  \widetilde{W}_{i,p}^{*}(0)=0.
\end{cases}
\ee
Since $u(y_{i,p})=\max\limits_{\overline{B_{r}(x_i)}}|u_{p}(x)|$, we have for $x\in \frac{B_{r}(x_i)\cap \Omega_p^+-y_{i,p}}{\e_{i,p}^{*}}$,
\[\widetilde{W}_{i,p}^{*}(x)\leq \widetilde{W}_{i,p}^{*}(0)=0,\quad \left|1+\frac{\widetilde{W}_{i,p}^{*}(x)}{p}\right|
=\left|\frac{u(y_{i,p}+\e_{i,p}^*x)}{u(y_{i,p})}\right|\leq1.\]
Fix any $R>1$, it follows from \eqref{3-45} that $B_R(0)\subset \frac{B_{r}(x_i)\cap \Omega_p^+-y_{i,p}}{\e_{i,p}^{*}}$ for $p$ large. Then we can repeat the proof of Lemma \ref{Lem25} to obtain that for large $p$,
\begin{align}
   |\nabla^m \widetilde{W}_{i,p}^{*}(x)|=O(R),\quad x\in B_R(0),\quad m=0,1,2,3.
\end{align}
Consequently, we can follow the proof of Proposition \ref{prop22} to conclude that up to a subsequence, $\widetilde{W}_{i,p}^{*}(x)\to W=-4\log(1+{|x|^2}/{8\sqrt{6}})$ in $C_{loc}^4(\mathbb{R}^4)$.

(3) The \eqref{3-35}-\eqref{3-36} can be proved by a similar argument as \eqref{765}.
\end{proof}

\begin{Lemma}\label{l5.1}
Fix any $x_i \in \s$. Given any $r\in(0,r_0)$, we have
\begin{align}\label{5.1}
    \liminf\limits_{p\to \infty} p\int_{B_r(x_i)}|u_p|^{p+1}\geq 64\pi^2 e.
\end{align}
Consequently,
\begin{align}\label{5.1-1}
    \liminf\limits_{p\to \infty} p\int_{\Omega}|u_p|^{p+1}\geq 64N\pi^2 e.
\end{align}

Moreover, if there exists $\varepsilon_1>0$ such that
\begin{align}\label{new 5.2}
    \liminf\limits_{p\to \infty} p\int_{B_r(x_i)\cap \Omega_p^+}|u_p|^{p+1}\geq \varepsilon_1 \;\text{ and }\; \liminf\limits_{p\to \infty} p\int_{B_r(x_i)\cap \Omega_p^-}|u_p|^{p+1}\geq \varepsilon_1,
\end{align}
then
\begin{align}\label{3-41}
    \liminf\limits_{p\to \infty} p\int_{B_r(x_i)}|u_p|^{p+1}\geq 128\pi^2 e.
\end{align}
\end{Lemma}

\begin{proof}
Take
$\psi\in C_{c}^\infty(B_{2r}(x_i))$ such that $0\leq \psi(x)\leq 1$ and
\be
\psi(x)=
\begin{cases}
  1, & \mbox{if } |x-x_i|\leq r \\
  0, & \mbox{if } |x-x_i|\in[\frac{4r}{3},2r).
\end{cases}\notag
\ee
Then $\widetilde{u}_{p}:=u_p\psi\in H_0^2(B_{2r}(x_i))$. Denote $R(r,\frac{4r}{3}):=\{x\in \Omega:r\leq|x-x_i|\leq\frac{4r}{3}\}$ for convenience.
Then by $(\ref{31})$, we obtain
\begin{align}
      &p\int_{B_{2r}(x_i)}|\dl \widetilde{u}_p|^2dx\notag\\
      =& p\int_{B_r(x_i)}|\dl u_{p}|^2dx+p\int_{R(r,\frac{4r}{3})}|\dl u_{p}|^2|\psi|^2dx
     + p\int_{R(r,\frac{4r}{3})}|u_{p}|^2|\dl\psi|^2dx\notag\\
     +&2p\int_{R(r,\frac{4r}{3})}|u_{p}\psi\dl u_{p}\dl\psi|dx+4p\int_{R(r,\frac{4r}{3})}|\langle \nabla u_{p},\nabla\psi\rangle|^2dx\notag\\
      +&4p\int_{R(r,\frac{4r}{3})}
      |u_{p}\dl\psi| |\langle \nabla u_{p},\nabla\psi\rangle| dx+4p\int_{R(r,\frac{4r}{3})}
      |\psi\dl u_{p}||\langle \nabla u_{p},\nabla\psi\rangle|dx
      \notag\\
      =&p\int_{B_r(x_i)}|\dl u_{p}|^2dx+o_{p}(1).\label{tt1}
\end{align}
Similarly,
\begin{align}
p\int_{B_{2r}(x_i)}|\widetilde{u}_{p}|^{p+1}dx
  =& p\int_{B_r(x_i)}|u_{p}|^{p+1}dx+p\int_{R(r,\frac{4r}{3})}|u_{p}|^{p+1}\psi^{p+1}dx\notag\\
  =&p\int_{B_r(x_i)}|u_{p}|^{p+1}dx+o_{p}(1)\nonumber\\
  \geq&64 \pi^2+o_p(1),\label{tt3}
\end{align}
where \eqref{3-35} is used to obtain the last inequality.
Thus, we deduce from (\ref{tt1})-(\ref{tt3}) and Lemma \ref{Lemm27} (i.e. apply it in $H_0^2(B_{2r}(x_i))$) that
\begin{align}\label{997}
    p\int_{B_r(x_i)}|\dl u_{p}|^2 dx
    =& p\int_{B_{2r}(x_i)}|\dl \widetilde{u}_{p}|^2dx +o_{p}(1)\notag\\
    \geq & \frac{p}{(p+1)p^{\frac{2}{p+1}}D_{p+1}^2}\left[p \int_{B_{2r}(x_i)}|\widetilde{u}_{p}|^{p+1}dx\right]^{\frac{2}{p+1}}+o_{p}(1) \notag\\
    \geq & (64\pi^2e+o_{p}(1))\left[64\pi^2+o_{p}(1)\right]^{\frac{2}{p+1}}+o_{p}(1) \notag\\
    \geq & 64\pi^2e+o_{p}(1)\quad\text{for $p$ large enough}.
\end{align}
On the other hand, by $\Delta^2u_p=|u_p|^{p-1}u_p$ and (\ref{31}), we have
\begin{align}\label{eq5.8}
    p\int_{B_r(x_i)}|\dl u_p|^2 dx=& p\int_{B_r(x_i)}|u_p|^{p+1} dx\notag\\
    &-p\int_{\pa B_r(x_i)} u_p\frac{\pa\dl u_p}{\pa\nu} dS+p\int_{\pa B_r(x_i)}\dl u_p \frac{\pa u_p}{\pa\nu} dS\notag\\
    =&p\int_{B_r(x_i)}| u_p|^{p+1}dx+o_p(1).
\end{align}
This proves \eqref{5.1}. Together with \eqref{enen}, we obtain \eqref{5.1-1}.

Now we assume \eqref{new 5.2} and prove \eqref{3-41}. To this goal, we need to apply a decomposition of $H_0^2(B_{2r}(x_i))$ from \cite{Weth2006dual}.
Consider the convex closed cone
$$\mathcal{K}=\{u\in H_0^2(B_{2r}(x_i))\;:\;u\geq0 \text{ a.e. in }\RN\}$$ and its dual cone $$\mathcal{K}^*=\{u\in H_0^2(B_{2r}(x_i))\;:\;\int_{B_{2r}(x_i)}\Delta u\Delta v\leq0 \text{ for all } v\in\mathcal{K}\}.$$
Since $\widetilde{u}_{p}=u_p\psi\in H_0^2(B_{2r}(x_i))$, it follows from \cite[Lemma 3.8]{Weth2006dual} that there exist $u_{p,1}\in\mathcal{K}\text{ and }u_{p,2}\in\mathcal{K}^*$ such that
\begin{align}\label{eq5.10}
    \widetilde{u}_p=u_{p,1}+u_{p,2}\,\, ,\quad u_{p,1}\geq\widetilde{u}_p^+\geq 0,\quad u_{p,2}\leq \widetilde{u}_p^-\leq 0,
\end{align}
\[\int_{B_{2r}(x_i)}\Delta {u}_{p,1}\Delta {u}_{p,2}=0.\]
Thus
\begin{align*}
    \int_{B_{2r}(x_i)}|\Delta \widetilde{u}_p|^2=\int_{B_{2r}(x_i)}|\Delta {u}_{p,1}|^2+\int_{B_{2r}(x_i)}|\Delta {u}_{p,2}|^2.
\end{align*}
By (\ref{new 5.2}) and (\ref{eq5.10}) and $\widetilde{u}_p=u_p$ in $B_r(x_i)$,
we have
\begin{align*}
    \liminf\limits_{p\to \infty} p\int_{B_{2r}(x_i)}|u_{p,1}|^{p+1}&\geq \liminf\limits_{p\to \infty} p\int_{B_{r}(x_i)}|u_p^+|^{p+1}\\
    &=\liminf\limits_{p\to \infty} p\int_{B_{r}(x_i)\cap\Omega_p^+}|u_p|^{p+1}\geq\varepsilon_1,
\end{align*}
and similarly
\[\liminf\limits_{p\to \infty} p\int_{B_{2r}(x_i)}|u_{p,2}|^{p+1}\geq\varepsilon_1.\]
Then by Lemma $\ref{Lemm27}$ again,
\begin{align*}
    p\int_{B_{2r}(x_i)}|\dl u_{p,i}|^2dx
    \geq & \frac{p}{(p+1)p^{\frac{2}{p+1}}D_{p+1}^2}\left[p \int_{B_{2r}(x_i)}|u_{p,i}|^{p+1} dx\right]^{\frac{2}{p+1}} \notag\\
    \geq & (64\pi^2e+o_{p}(1))(\varepsilon_1)^{\frac{2}{p+1}} \notag\\
    \geq & 64\pi^2e+o_{p}(1)\quad \text{for $p$ large enough, $i=1,2.$}
\end{align*}
Hence,
\begin{align}\label{new 5.14}
    \liminf\limits_{p\rt\iy}p\int_{B_{2r}(x_i)} |\Delta\widetilde{u}_p|^2\geq128\pi^2e.
\end{align}
Finally, \eqref{3-41} follows from (\ref{tt1}), (\ref{eq5.8}) and (\ref{new 5.14}).
\end{proof}

\section{Proof of Theorem \ref{Theorem 1.1} and Theorem \ref{th12}}

\label{section-5}

In this section, we assume that the domain $\Omega$ satisfies the condition (G) and let $(u_p)_{p>1}$ be the least energy nodal solutions of \eqref{PP}, the energy estimate of which was already studied in Section 2. Now we want to apply those results of Sections 2-3 to prove Theorem \ref{Theorem 1.1} and Theorem \ref{th12}.

We use the same notations in Sections 2-\ref{section-3}. Recall that \[u_p(x_p^+)=\Vert u_p^+ \Vert_{\infty}=\Vert u_p \Vert_{\infty}\geq -u_p(x_p^-)=\Vert u_p^- \Vert_{\infty},\]  and $x^+=\lim\limits_{p\rt\infty}x_p^+=x_1\in\s$.

\begin{Proposition}\label{prop666}
\begin{align}\label{use4.3}
    \limsup\limits_{p \to \infty} \left\Vert u_{p} \right\Vert_{\infty}\leq \sqrt{e}.
\end{align}

\begin{proof}
Thanks to Proposition \ref{prop4.1} and Corollary \ref{cor31}, we have
\begin{align}
   64\pi^2e+o(1)=&p \int_{\Omega}|u_{p}^+|^{p+1} \nonumber\\
     =&u_{p}(x_{p}^+) ^2 \int_{\frac{\Omega_p^+-x_p^+}{\varepsilon_p^+}}\left|1+\frac{W_{p^+}(z)}{p}\right|^{p+1} dz\notag\\
     \geq & u_{p}(x_{p}^+) ^2 (\int_{\mathbb{R}^4}e^W+o(1))
     =  u_{p}(x_{p}^+)^2 (64\pi^2+o(1)),
\end{align}
which implies (\ref{use4.3}).
\end{proof}
\end{Proposition}

\begin{Proposition}\label{prop51}
$\#\s=2$ and so $\s=\{x_1, x_2\}$ with $x_1=x^+=\lim\limits_{p\to\infty}x_p^+$.
\end{Proposition}

\begin{proof}
Proposition \ref{Lem26} and Lemma \ref{l5.1} imply
\begin{align*}
   128\pi^2e=\lim_{p\to\infty}p\int_{\Omega} \left|u_p\right|^{p+1}\geq 64 N\pi^2 e,
\end{align*}
so $N=\#\s\leq 2$.
If $\s=\{x_1\}$, then \eqref{3-34-2} gives
\[128\pi^2e=\frac{(\gamma_1^+-\gamma_1^-)^2}{64\pi^2},\]
i.e.
$|\gamma_1^+-\gamma_1^-|=64\pi^2 \sqrt{2e}.$

On the other hand, for any $\rho\in (0,r_0)$,
we see from \eqref{31} that
\begin{align}\label{4-4-2}p\int_{\Omega}|u_{p}^\pm|^{{p}+1} dx
=p\int_{B_\rho(x_1)}|u_{p}^\pm|^{{p}+1} dx+o_p(1),\end{align}
and so Proposition \ref{prop4.1} yields
\begin{align}\label{4-4-3}
    \lim\limits_{p\to \infty} p\int_{B_{\rho}(x_1)}|u_{p}^{\pm}|^{{p}+1} dx=\lim\limits_{p\to \infty} p\int_{\Omega}|u_{p}^{\pm}|^{{p}+1} dx=64 \pi^2 e.
\end{align}
Since Proposition $\ref{prop666}$ implies
\begin{align}
    p\int_{B_{ \rho}(x_1)}|u_{p}^{\pm}|^{{p}+1} dx \leq \sqrt{e} p\int_{B_{ \rho}(x_1)}|u_{p}^{\pm}|^{{p}} dx=\sqrt{e} p\int_{B_{ \rho}(x_1)\cap\Omega_p^\pm}|u_{p}|^{{p}} dx,
\end{align}
we conclude from the definition \eqref{38} of $\gamma_{i}^{\pm}$, i.e.
\[\gamma_1^\pm=\lim\limits_{\rho\rt 0}\lim\limits_{p \to \iy} p\int_{B_{\rho }(x_1)\cap \Omega_p^\pm}|u_{p}(x)|^{p}dx,\]
that $\gamma_1^{\pm} \geq 64 \pi^2 \sqrt{e}$.
Furthermore, by H\"{o}lder inequality and \eqref{4-4-3},
\begin{align}\label{gamma1}
    \gamma_1^\pm\leq
    \lim\limits_{p\to \infty} p^{\frac{1}{p+1}}\left(p\int_{B_{ \rho}(x_1)}|u_{p}^\pm|^{p+1}dx\right)^{\frac{p}{p+1}}\left|\Omega\right|^{\frac{1}{p+1}} \leq 64\pi^2 e,
\end{align}
so $64 \pi^2 \sqrt{e}\leq \gamma_1^{\pm}\leq 64 \pi^2e$, clearly a contradiction with
$|\gamma_1^+-\gamma_1^-|=64\pi^2 \sqrt{2e}.$
Therefore, $\#\s=2$.
\end{proof}

\begin{Proposition}\label{prop5.2}
For any $\rho\in (0,r_0)$, we have
\begin{align*}
    \lim\limits_{p\to \infty} p\int_{B_{\rho}(x_1)}|u_{p}^+|^{p+1}dx&=
    \lim\limits_{p\to \infty} p\int_{B_{\rho}(x_2)}|u_{p}^-|^{p+1}dx=64\pi^2 e,\\
    \lim\limits_{p\to \infty} p\int_{B_{\rho}(x_1)}|u_{p}^-|^{p+1}dx&=
    \lim\limits_{p\to \infty} p\int_{B_{\rho}(x_2)}|u_{p}^+|^{p+1}dx=0.
\end{align*}
Moreover, $\gamma_1^+=\gamma_2^-=64\pi^2\sqrt{e}$ and $\gamma_1^-=\gamma_2^+=0$.
\end{Proposition}

\begin{proof}
By $\s=\{x_1, x_2\}$, Proposition \ref{Lem26} and Lemma \ref{l5.1}, we have
\begin{align}\label{5.18}
    \lim\limits_{p\to \infty}\int_{B_{\rho}(x_i)}|u_p|^{p+1}=64\pi^2 e,\quad i=1,2.
\end{align}
Besides, by a similar argument as \eqref{4-4-2}-\eqref{4-4-3}, we see from
Proposition \ref{prop4.1} that
\begin{align}
   64\pi^2 e &=\lim\limits_{p\to \infty} p\int_{\Omega}|u_{p}^{\pm}|^{p+1}dx\notag\\
   &=
\lim\limits_{p\to \infty} p\int_{B_{\rho}(x_1)}|u_{p}^{\pm}|^{p+1}dx
    +\lim\limits_{p\to \infty}p\int_{B_{\rho}(x_2)}|u_{p}^{\pm}|^{p+1}dx. \label{5.19}
\end{align}
Since $x_1=x^+=\lim_{p\to\infty}x_p^+$,
it follows from Lemma \ref{Lem31} (3) that
\begin{align}
    \lim\limits_{p\to \infty} p\int_{B_{\rho}(x_1)}|u_{p}^{+}|^{p+1}dx
    =\lim\limits_{p\to \infty} p\int_{B_{\rho}(x_1)\cap\Omega_p^+}|u_{p}|^{p+1}dx
    \geq 64\pi^2.
\end{align}
If up to a subsequence, there exists $\varepsilon_1>0$ such that
\begin{align}\label{5.20}
    \lim\limits_{p\to \infty} p\int_{B_{\rho}(x_1)}|u_{p}^{-}|^{p+1}dx
    =\lim\limits_{p\to \infty} p\int_{B_{\rho}(x_1)\cap\Omega_p^-}|u_{p}|^{p+1}dx\geq
    \varepsilon_1,
\end{align}
Then Lemma \ref{l5.1} implies
\[
    \lim\limits_{p\to \infty} p\int_{B_{\rho}(x_1)}|u_{p}|^{p+1}dx\geq
    128\pi^2 e,
\]
which is a contradiction with (\ref{5.18}).
Thus,
\begin{align}
    \lim\limits_{p\to \infty} p\int_{B_{\rho}(x_1)}|u_{p}^+|^{p+1}dx=
    64\pi^2 e,\quad\lim\limits_{p\to \infty} p\int_{B_{\rho}(x_1)}|u_{p}^-|^{p+1}dx=
    0,
\end{align}
and then it follows from \eqref{5.19} that
\begin{align}
     \lim\limits_{p\to \infty} p\int_{B_{\rho}(x_2)}|u_{p}^+|^{p+1}dx=
   0,\quad\lim\limits_{p\to \infty} p\int_{B_{\rho}(x_2)}|u_{p}^-|^{p+1}dx=64\pi^2 e.
\end{align}
Consequently, similarly as \eqref{gamma1} we have
\begin{align}
    0\leq \gamma_1^- \leq \lim\limits_{p\to \infty} p^{\frac{1}{p+1}}\left(p\int_{B_{\rho}(x_1)}|u_{p}^-|^{p+1}dx\right)^{\frac{p}{p+1}}\left|\Omega\right|^{\frac{1}{p+1}} =0,
\end{align}
i.e. $\gamma_1^-=\gamma_2^+=0$. Since \eqref{3-34} and \eqref{5.18}
imply $|\gamma_i^+-\gamma_i^-|=64\pi^2\sqrt{e}$ for $i=1,2$, we finally obtain
$\gamma_1^+=\gamma_2^-=64\pi^2\sqrt{e}$.
\end{proof}

\begin{Lemma}\label{Lemma last}
For any $\rho\in (0,r_0/2)$, we have
\[\liminf_{p\to\infty} \inf_{x\in B_\rho(x_1)}u_p(x)\geq 0,\quad
\limsup_{p\to\infty} \sup_{x\in B_\rho(x_2)}u_p(x)\leq 0.\]
\end{Lemma}

\begin{proof}
Let $\widetilde{G}(x,y)$ is the Green function of $\Delta^2$ on $B_{2\rho}(x_1)$ under Dirichlet boundary conditions, i.e.
\[\begin{cases}
\Delta_x^2 \widetilde{G}(x,y)=\delta(x-y) \quad &\hbox{in}\;B_{2\rho}(x_1), \\
\widetilde{G}(x,y)=\frac{\partial G(x,y)}{\partial \nu}=0 \ \ &\hbox{for}\;x\in\partial B_{2\rho}(x_1).
\end{cases}\]
Then for any $x\in B_{\rho}(x_1)$, we have
\begin{align}\label{s5}
    u_p(x)=&\int_{B_{2\rho}(x_1)}\widetilde{G}(x,y)|u_p^+(y)|^{p} dy-\int_{B_{2\rho}(x_1)}\widetilde{G}(x,y)|u_p^-(y)|^{p} dy\nonumber\\
    +&\int_{\partial B_{2\rho}(x_1)}\left(u_p(y)\frac{\partial\Delta_y \widetilde{G}(x,y)}{\partial \nu}-\Delta_y \widetilde{G}(x,y)\frac{\partial u_p(y)}{\partial \nu}\right) dS.
\end{align}
Since $|x-y|\geq \rho$ for $x\in B_{\rho}(x_1)$ and $y\in \partial B_{2\rho}(x_1)$,
we see from Lemma $\ref{Lemma 2.1}$ and \eqref{31} that
\begin{align}\label{s3}
    \left|\int_{\partial B_{2\rho}(x_1)} u_p(y)\frac{\partial\Delta_y \widetilde{G}(x,y)}{\partial \nu} dS\right|
    \leq \frac{C}{p}\int_{\partial B_{2\rho}(x_1)}
    \frac{dS}{|x-y|^3}\leq \frac{C}{p},
\end{align}
and similarly
\[
    \left|\int_{\partial B_{2\rho}(x_1)}\Delta_y \widetilde{G}(x,y)\frac{\partial u(y)}{\partial \nu}\ dS\right|\leq \frac{C}{p}.
\]
By Lemma \ref{Lemma 2.1} we also have
\[
    \int_{B_{2\rho}(x_1)}\widetilde{G}(x,y)|u_p^+(y)|^{p} dy\geq -C \int_{\Omega} |u_p^+(y)|^{p} dy \geq -\frac{C}{p}.
\]
On the other hand, recall \eqref{ggg} that
\[\int_{B_{2\rho}(x_1)}|\widetilde{G}(x,y)|^{p+1}dy\leq C(p+1)^{p+2},\]
so we deduce from H\"{o}lder inequality and Proposition \ref{prop5.2} that
\begin{align}\label{s1}
    &\left|\int_{B_{2\rho}(x_1)}\widetilde{G}(x,y)|u_p^-(y)|^{p} dy\right|\notag\\
    \leq & \left(C(p+1)^{p+2}\right)^{\frac1{p+1}} p^{-\frac{p}{p+1}}
    \left(p\int_{B_{2\rho}(x_1)}|u_p^-(y)|^{p+1}dy\right)^{\frac{p}{p+1}}
    =o_p(1).
\end{align}
Inserting (\ref{s3})-(\ref{s1}) into \eqref{s5} leads to $\liminf_{p\to\infty} \inf_{x\in B_\rho(x_1)}u_p(x)\geq 0$. The $\limsup_{p\to\infty} \sup_{x\in B_\rho(x_2)}u_p(x)\leq 0$ can be proved similarly.
\end{proof}

Recall Section 1 that $x_p^-\in\Omega_p^-$ such that $u_p(x_p^-)=-\|u_p^-\|_{\infty}$. Now we can show that $x^-:=\lim\limits_{p\rt\infty}x_p^-\in\s$ as follows.

\begin{Lemma}\label{prop5.3}
Recalling $x_1=x^+=\lim\limits_{p\to\infty}x_p^+$, we have $x_2=x^-:=\lim\limits_{p\rt\infty}x_p^-$ and so $x^+\neq x^-$.
\end{Lemma}

\begin{proof}
Recall Lemma \ref{Lem31}: Given $r\in (0,r_0/2)$, we define
$y_{2,p}\in \overline{B_{r}(x_2)}$ such that
\begin{align}\label{4-17}
    |u_p(y_{2,p})|=\max\limits_{B_{r}(x_2)}\left|u_p(x)\right|,
\end{align}
then $p| u_p(y_{2,p})|^{p-1} \to \infty$ as $p \to \infty$ and $y_{2,p}\to x_2$, so $|u_{p}(y_{2,p})|\geq \frac12$ for $p$ large. This, together with Lemma \ref{Lemma last}, implies $u_p(y_{2,p})\leq -\frac12$, i.e. $y_{2,p}\in \Omega_p^-$. Thus
\[u_p(x_p^-)=-\|u_p^-\|_{\infty}\leq u_p(y_{2,p})\leq -\frac{1}{2},\quad \text{for $p$ large}.\] From here and
Lemma \ref{Lemma last}, we get $x^-\neq x_1=x^+$. If $x^{-}\neq x_2$, i.e. $x^{-}\notin \s$, then \eqref{31} implies $1/2\leq |u_p(x_p^-)|\leq \frac{C}{p}\to 0$, a contradiction. Thus $x^-=x_2$. From here and \eqref{4-17}, we also have $u_p(x_p^-)=u_p(y_{2,p})=-\|u_p^-\|_{\infty}$, i.e. we can take $y_{2,p}=x_p^-$ actually. Then Lemma \ref{Lem31} implies
\begin{align*}
    \frac{\Omega-x_p^-}{\varepsilon_p^-}\to \mathbb{R}^4 \quad, \quad
    \frac{\Omega_p^--x_p^-}{\varepsilon_p^-}\to \mathbb{R}^4,
\end{align*}
and $W_{p^-} \to W$ in $C^4_{loc}(\mathbb{R}^4)$ as $p \to \infty$.
\end{proof}

\begin{corollary}\label{cor}
$\lim\limits_{p\to \infty}u_p(x_p^+)=\sqrt{e}$ and $\lim\limits_{p\to \infty}u_p(x_p^-)=-\sqrt{e}$.
\end{corollary}

\begin{proof}
By Proposition $\ref{prop5.2}$, it follows that
\begin{align}
     \liminf\limits_{p\to \infty}u_p(x_p^+) \, p\int_{B_{\rho}(x_1)}|u_{p}^+|^{p}dx &\geq \liminf\limits_{p\to \infty} p\int_{B_{\rho}(x_1)}|u_{p}^+|^{p+1}dx=
   64\pi^2 e.\\
   \lim_{\rho\to 0}\lim\limits_{p\to \infty} p\int_{B_{\rho}(x_1)}|u_{p}^+|^{p}dx &=\gamma_1^+=64\pi^2\sqrt{e},
\end{align}
so $\liminf\limits_{p\to \infty}u_p(x_p^+)\geq \sqrt{e}$.
Together with Proposition \ref{prop666}, we obtain $\lim\limits_{p\to \infty}u_p(x_p^+)= \sqrt{e}$.
The $\lim\limits_{p\to \infty}u_p(x_p^-)= -\sqrt{e}$ can be proved similarly.
\end{proof}

We are in the position to finish the proof of Theorem \ref{Theorem 1.1} and Theorem \ref{th12}.

\begin{proof}[Proof of Theorem \ref{Theorem 1.1} and Theorem \ref{th12}] Remark that \eqref{new 1,5}-\eqref{new 1,5-2} are proved in Proposition \ref{Lem26}, Proposition \ref{prop4.1} and Corollary \ref{cor}. Furthermore,
by Proposition $\ref{prop3.2}$, Proposition $\ref{prop5.2}$ and Lemma $\ref{prop5.3}$, we have
\begin{align}
    p u_p\to 64\pi^2\sqrt{e}(G(\cdot,x^+)-G(\cdot,x^-))
\end{align}
in $C_{loc}^4(\overline{\Omega}\setminus\{x^+,x^-\})$ as $p\rt\infty$. The convergence of $W_{p^+}$ and $W_{p^-}$ are proved in Proposition \eqref{prop22} and Lemma \ref{prop5.3} respectively. Therefore, it suffices to prove $(x^+, x^-)$ satisfies the equation \eqref{801}, which can be proved by a standard argument using
the Pohozaev identity in
$\Omega \setminus B_{\rho}(x^+)$ and $\Omega \setminus B_{\rho}(x^-)$; see e.g. \cite{Grossi2013Poincare} or \cite{wei}. We omit the details here. This completes the proof.
\end{proof}

\subsection*{Acknowledgements} The research of Z. Chen was supported by NSFC (No. 12222109, 12071240).

\subsection*{Declarations} {\bf Conflict of interest} There is no conflict of interest between the authors and all data generated or analysed
during this study are included in this article.

\end{document}